\RequirePackage{fix-cm}
\documentclass[smallcondensed]{svjour3}     
\smartqed  
\usepackage{graphicx}
\usepackage[T1]{fontenc}
\usepackage{lmodern}
\usepackage{textcomp}
\usepackage{latexsym}
\usepackage{comment}
\usepackage{amsmath}
\usepackage{amssymb}
\usepackage{amsfonts}
\usepackage{listings}
\newtheorem{thm}{Theorem}
\newtheorem{defi}[thm]{Definition}
\newtheorem{lem}[thm]{Lemma}
\newtheorem{prop}[thm]{Proposition}
\newtheorem{cor}[thm]{Corollary}
\newtheorem{rem}[thm]{Remark}
\begin{document}

\title{Asymptotic Bayesian Generalization Error in Latent Dirichlet Allocation and Stochastic Matrix Factorization
}
\subtitle{}

\titlerunning{Bayesian Generalization Error in Latent Dirichlet Allocation}        

\author{Naoki Hayashi*         \and
        Sumio Watanabe 
}


\institute{N. Hayashi \at
              NTT DATA Mathematical Systems Inc.,
              1F Shinano-machi-Renga-kan Shinano-machi 35, Shinjuku-Ku, Tokyo, 160-0016, JAPAN \\
              Tel.: +813-3358-1809, Fax: +813-3358-1727\\
              \email{hayashi@msi.co.jp}           
           \and
           N. Hayashi \and S. Watanabe \at
              Tokyo Institute of Technology
}

\date{Received: date / Accepted: date}

\maketitle

\begin{abstract}
Latent Dirichlet allocation (LDA) is useful in document analysis, image processing, and many information systems; 
however, its generalization performance has been left unknown 
because it is a singular learning machine to which regular statistical theory can not be applied.

Stochastic matrix factorization (SMF) is a restricted matrix factorization in which matrix factors are stochastic; the column of the matrix is in a simplex. SMF is being applied to image recognition and text mining.
We can understand SMF as a statistical model by which a stochastic matrix of given data is represented by a product of two stochastic matrices, whose generalization performance has also been left unknown because of non-regularity.

In this paper, by using an algebraic and geometric method, 
we show the analytic equivalence of LDA and SMF,
both of which have the same real log canonical threshold (RLCT), resulting in that they asymptotically have the same Bayesian generalization error and the same log marginal likelihood. Moreover, we derive the upper bound of the RLCT and prove that it is smaller than the dimension of the parameter divided by two, hence the Bayesian generalization errors of them are smaller than those of regular statistical models.
\end{abstract}

\section{Introduction}
\label{sec-intro}
\subsection{Latent Dirichlet Allocation}
The topic model \cite{Gildea1999topic} is a ubiquitous learning machine used in many research areas, including text mining \cite{David2003LDA,Griffiths2004LDAGS}, computer vision \cite{Li2005LDA4CV}, marketing research \cite{Tirunillai2014LDA4MR}, and geology \cite{Yoshida2018LDA4Geo}.
Latent Dirichlet allocation (LDA) \cite{David2003LDA} is one of the most popular Bayesian topic models.
It has been devised for text analysis, and it can utilize information in documents by defining the topics of the words. The topics are formulated as one-hot vectors subject to categorical distributions which are different for each document (Fig. \ref{fig:topic-model}). The standard inference algorithms, such as Gibbs sampling \cite{Griffiths2004LDAGS} and the variational Bayesian approximation \cite{David2003LDA}, require the
appropriate number of the topics to be set. Different topics are inferred as the same thing if the chosen number of topics is too small; that is, LDA suffers from underfitting. On the other hand, if the chosen number of topics is too large, the model suffers from overfitting on the training data.
In practical applications, the optimal number of topics of the ground truth is unknown; thus, researchers and practitioners face a situation in which the number of topics they set may be larger than the optimal one. Since such cases frequently appear in practical model selection problems, clarifying the behavior of the generalization error is necessary as a theoretical foundation of statistical model selection processes.
However, the mathematical property of LDA has not yet been clarified, because
it is a singular learning machine to which the regular statistical theory can not be applied.

\begin{figure*}[h]
\begin{minipage}{0.5\hsize}
\begin{flushleft}
\includegraphics[width=6cm,height=4.5cm]{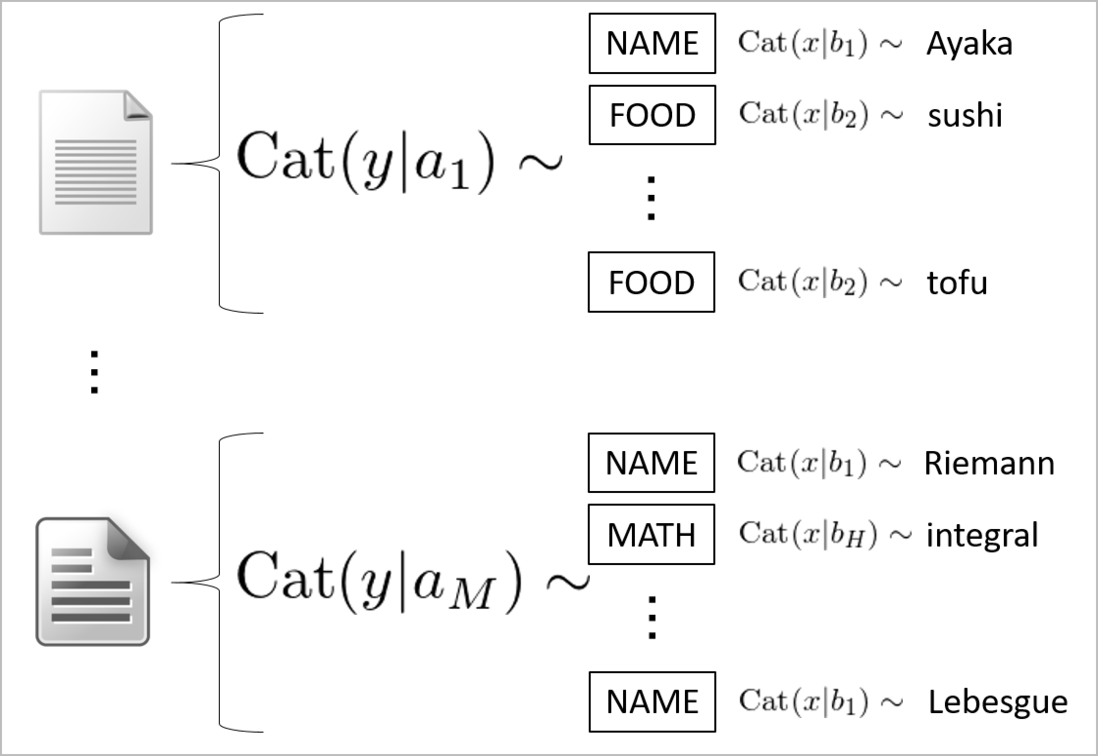}
\label{fig:topic-model}
\end{flushleft}
\end{minipage}
\begin{minipage}{0.5\hsize}
\begin{flushright}
\includegraphics[width=6cm,height=4.5cm]{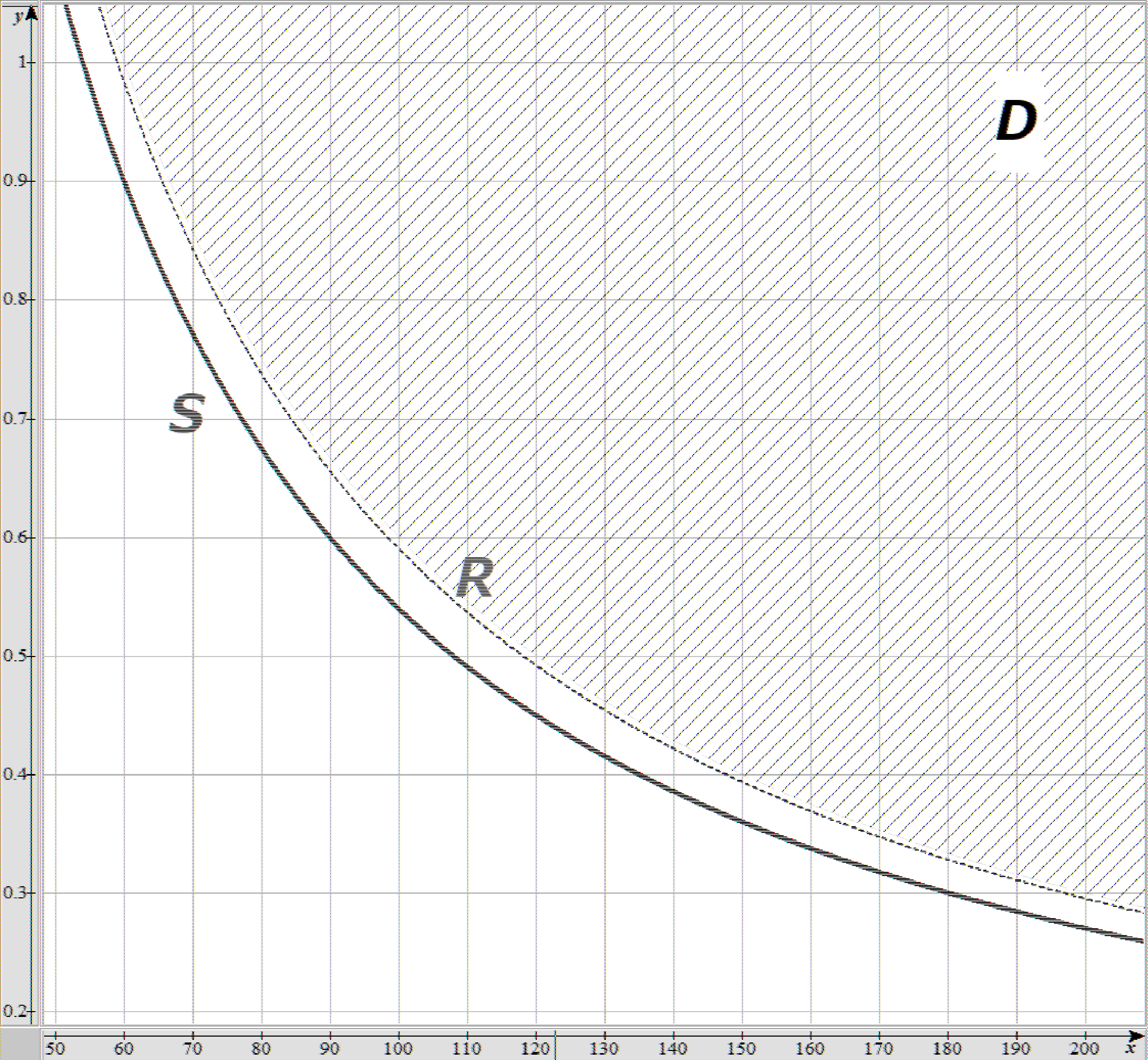}
\label{fig:learning-curve}
\end{flushright}
\end{minipage}
\caption{(Left fig.) This figure gives an overview of LDA. The categorical distributions $\mathrm{Cat}$ that depend on the documents, where $M$ and $H$ are the numbers of the documents and topics, respectively. Words in the uppercase such as NAME, FOOD, and MATH are topics. There are categorical distributions that are different for each topic; the words (Ayaka, sushi, integral, ...) are generated from them. Hence, we can explain LDA as a {\it mixture} of categorical mixture models. This model has a hierarchical structure.
(Right fig.) This figure shows the theoretical learning curve of LDA/SMF. In this paper, we prove that the learning curve of LDA/SMF, when Bayesian inference is applied, behaves like $S$ in the figure; i.e., the generalization error becomes strictly lower than those of regular statistical models and the maximum likelihood or posterior method.
$D$ is the domain that includes the learning curve of LDA/SMF when the maximum likelihood or posterior method is applied.
$R$ is the learning curve when the model is regular and its parameter dimension is the same as that of LDA/SMF.}
\end{figure*}

\subsection{Stochastic Matrix Factorization}
Matrix factorization (MF) is also used in machine learning frequently. MF decomposes the data matrix into a product of two matrices and discovers hidden structures or patterns, hence it has been experimentally used for knowledge discovery in many fields. However, MF has no guarantee of reaching the unique factorization, and it is sensitive to the initial value of the numerical calculation. This non-uniqueness interferes with data-driven inference and interpretations of the results. Besides, the sensitivity to the initial value causes the factorization result to have low reliability. From the viewpoint of data-based prediction,
this instability may lead to incorrect predictions. To improve interpretability, non-negative matrix factorization (NMF) \cite{Paatero,Lee} has been devised; it is a restricted MF wherein the elements of the matrix are non-negative. Thanks to the non-negativity constraint, the extracted factors are readily interpretable, therefore NMF is frequently used for extracting latent structures and patterns, for instance, image recognition \cite{Lee}, audio signal processing \cite{schmidt2006single} and consumer analysis \cite{Kohjima}. However, the non-uniqueness property and initial value sensitivity have not yet been settled.

Stochastic matrix factorization (SMF) was devised by Adams \cite{Adams2016SMF}; it can be understood as a restriction of NMF in which at least one matrix factor is ``stochastic'': the elements of the matrix factors are non-negative and the sum of the elements in a column is equal to 1. A ``stochastic'' matrix is defined as a matrix with at least one ``stochastic'' column. By making two further assumptions, Adams proved the uniqueness of the results of SMF \cite{Adams2016FMM,Adams2016SMF}. For a statement of these two conditions, let us consider a data matrix $X$ whose size is $M \times N$ and factor matrices $A$ and $B$ which are ``stochastic'' and whose sizes are $M \times H$ and $H \times N$, respectively. $H$ might be the rank of $X$ but the ``stochastic'' condition makes this determination non-trivial. In other words, SMF can be viewed as a method that finds a factor matrices pair $(A,B)$ such that $X=AB$ for a given $X$ and $H$. The non-uniqueness property has been paraphrased as the existence of $H \times H$ regular matrix $P \ne I_H$ such that
\begin{equation}
\label{non-uniq}
X=A P P^{-1} B,
\end{equation}
where $I_H$ is an $H \times H$ identity matrix.
Thus, uniqueness means that eq.(\ref{non-uniq}) is attained if and only if $P=I_H$. Adams assumed that 
\begin{equation}
\label{AdamsAssump1}
AP \geqq 0 \ \mbox{\rm and} \ P^{-1}B \geqq 0,
\end{equation}
i.e., the elements of $AP$ and $P^{-1}B$ are non-negative, and $P^{-1}B =: (b'_{kj})$ satisfies
\begin{equation}
\label{AdamsAssump2}
\sum_{k=1}^H b'_{kj} =1 \ \mbox{\rm or} \ \sum_{j=1}^N b'_{kj} =1.
\end{equation}
Adams claimed that these assumptions are ``natural'' and
applied SMF to image recognition (the same problem analyzed in \cite{Lee}) and text mining\cite{Adams2016SMF}.

It is emphasized that, in this paper,
 we consider the case ($\alpha$) when all matrix factors are stochastic rather than the case ($\beta$) when at least one matrix factor is stochastic. Adams proved that SMF reaches a unique factorization under some assumptions in case ($\beta$) \cite{Adams2016FMM,Adams2016SMF}. However, in general, stochastic matrices do not satisfy these assumptions. The term ``stochastic matrix'' usually means case ($\alpha$).
In addition, 
it is not clear whether Adam's assumptions (\ref{AdamsAssump1}) and (\ref{AdamsAssump2}) are mathematically ``natural''. For simplicity, we call the model an SMF in case ($\alpha$).

The MF methods described so far, including SMF, are understood as a deterministic procedure. As will be shown later, for hierarchical learning machines such as MF, we study probabilistic procedures, because
Bayesian inference has higher predictive accuracy than deterministic methods or maximum likelihood estimation. The same is also true regarding the accuracy of the discovered knowledge.
Moreover, the probabilistic or statistical view gives a wider application. Indeed, Bayesian NMF \cite{Virtanen2008bayesian,Cemgil} has been applied to image recognition \cite{Cemgil}, audio signal processing \cite{Virtanen2008bayesian},
and recommender systems \cite{Bobadilla2018recommender}. From a statistical point of view, the data matrices are random variables subject to the true distribution. Sometimes, MF is studied in
a case when only one target matrix is decomposed, however, the factorization of a set of independent matrices should be studied because the target matrices are often obtained daily, monthly, or in different places \cite{Kohjima}.
More importantly, as proved later, the SMF has the same learning curve as LDA; if the Bayesian generalization error in SMF has been clarified, then that of LDA is also determined. That is why the decomposition of a set of matrices is considered to be a statistical inference in this research.

\subsection{Bayesian Learning Theory}
A learning machine or a statistical model is said to be regular if
a map from a parameter to a probability density function is injective
and if the likelihood function can be approximated by a Gaussian function
for the sufficiently large sample size.
It has been proved that, if a statistical model is regular and if the true distribution is realizable by a statistical model, then the expected generalization error
$\mathbb{E}[G_n]$ defined by the Kullback-Leibler distance of the
true and estimated probability density functions $q(x)$ and $p^*(x)$,
$$G_n=\int q(x) \log \frac{q(x)}{p^*(x)} dx$$
is asymptotically equal to $d/(2n)$,
where $d$ and $n$ are the dimension of the parameter and the sample size,
respectively\cite{SWatanabeBookMath}.
However, the learning machine used in either LDA or SMF is not regular because the map from a parameter to a probability
density function is not one-to-one. Such a model is called a singular learning machine, whose theoretical generalization error has been unknown, resulting in that that we cannot confirm the correctness of the results of numerical experiments.

There are many practical singular learning machines, for example,
Gaussian mixture models, reduced rank regressions, neural networks, hidden Markov models, and Boltzmann machines. Both NMF and SMF are also statistically singular.
It is proved that the expected generalization error of a singular learning machine in Bayesian learning has an asymptotic expansion,
\begin{equation}\label{EGnRLCT}
\mathbb{E}[G_n]=\frac{\lambda}{n} +o\left(\frac{1}{n}\right),
\end{equation}
where $\lambda$ is the real log canonical threshold (RLCT), which is a birational invariant in algebraic geometry \cite{Watanabe1,WatanabeAIC,SWatanabeBookE}
determined by a true distribution and a statistical model. The RLCT is also called the learning coefficient \cite{Drton,Aoyagi2}, as it is the coefficient of the main term in the above expansion. In addition, the negative log Bayesian marginal likelihood $F_n$ can be asymptotically expanded as
$$F_n=nS_n+\lambda \log n +o_p(\log n),$$
where $S_n$ is the empirical entropy.
Note that RLCTs are different from the usual log canonical thresholds \cite{Hironaka} since the real field is not algebraically closed and the usual log canonical threshold is defined on an algebraically closed field such as the complex field. Thus, we cannot directly apply the research results in algebraically closed fields to
machine learning and statistics. The RLCTs were clarified in
mixture models \cite{Yamazaki1},
reduced rank regressions \cite{Aoyagi1},
three-layered neural networks \cite{Watanabe2},
naive Bayesian networks \cite{Rusakov2005asymptotic},
Markov models \cite{Zwiernik2011asymptotic},
and NMFs \cite{nhayashi2,nhayashi5},
by using a resolution of singularities \cite{Hironaka,Atiyah1970resolution}. Finding the RLCTs means deriving the theoretical value of the generalization errors and negative log marginal likelihoods. Besides, a statistical model selection method, called singular Bayesian information criterion (sBIC), that
utilizes RLCTs to estimate the negative log Bayesian marginal likelihood has also been proposed \cite{Drton}. The exchange Markov chain Monte Carlo has a dominant parameter called temperature. A setting method for the temperature by using RLCTs has also been studied \cite{Nagata2008asymptotic}. Thus, clarification of the RLCTs for learning machines is important from not only
theoretical but also a practical viewpoint.

\subsection{Main Contribution and Structure of Paper}
The main contributions of the present paper are summarized as follows:
\begin{enumerate}
  \item The asymptotic form of the Bayesian generalization error and the 
negative log marginal likelihoods in LDA and SMF are theoretically obtained. 
  \item  It is proved that LDA and SMF are equivalent to each other from an algebraic and geometrical
point of view. 
\end{enumerate}

Below, we study the theoretical generalization error in LDA when Bayesian learning is applied. We theoretically derive an upper bound of the RLCT of SMF, with which we can derive an upper bound of the expected Bayesian generalization error in LDA and SMF.
We would like to emphasize that the bound cannot be immediately proved in the same way as with NMF and other learning machines. There is no standard method to find the RLCT to a given family of functions; instead, researchers study RLCTs by developing novel methods for each learning machine.

This paper consists of five parts.
The second section describes the framework of Bayesian inference and the upper bound of the RLCT in LDA and SMF (Main Theorem).
The third section mathematically prepares basic propositions for the proof of the Main Theorem.
In the fourth section, we give the proof of the Main Theorem.
The fifth section describes a theoretical application of the Main Theorem to Bayesian learning.

\section{Framework and Main Result} 
 
Here, we explain the framework of Bayesian learning and of analyzing the RLCTs of learning machines and then introduce the main result of this paper. 

\subsection{Framework of Bayesian Learning} 
 
First, we explain the general theory of Bayesian learning. 
Let $q(x)$ and $p(x|\theta)$ be probability density functions on a finite-dimensional real Euclidean space, where $\theta$ is a parameter. 
In learning theory, $q(x)$ and $p(x|\theta)$ respectively represent the true distribution and a learning machine given $\theta$.
A probability density function $\varphi(\theta)$ whose domain is a set of parameters is called a prior. 
Let $X^n=(X_1,X_2,...,X_n)$ be a set of random variables that are independently subject to $q(x)$, where $n$ and 
$X^n$ are the sample size and training data respectively. 
The probability density function of $\theta$ defined by
\[
\psi(\theta|X^n):=\frac{1}{Z(X^n)}\varphi(\theta) \prod_{l=1}^n p(X_l|\theta)
\]
is called the posterior, where $Z(X^n)$ is a normalizing constant determined by the condition $\int \psi(w|X^n)=1$:
$$Z(X^n)=\int \varphi(\theta) \prod_{l=1}^n p(X_l|\theta) d\theta.$$
This is called the marginal likelihood or partition function. The Bayesian predictive distribution is defined by
\[
p^*(x):=p(x|X^n)=\int p(x|\theta)\psi(\theta|X^n)d\theta. 
\]
Bayesian inference/learning means inferring that the predictive distribution is the true distribution. 

Bayesian inference is statistical; hence, its estimation accuracy should be verified. There are mainly two criteria for this verification.
The first is the negative log marginal likelihood:
$$F_n:=-\log Z(X^n).$$
This is also called the free energy or the stochastic complexity\cite{SWatanabeBookE}.
The second is the generalization error $G_n$. It is defined by the Kullback-Leibler divergence of the true distribution $q(x)$ and
the predictive one $p(x|X^n)$:
\[
G_n:=\int q(x)\log\frac{q(x)}{p(x|X^n)}dx.
\]
Note that $F_n$ and $G_n$ are functions of $X^n$ hence they are also random variables. The expected value of $G_n$ for the
overall training data $\mathbb{E}[G_n]$ is called the expected generalization error. 
Let us assume there exists at least one parameter $\theta_0$ that satisfies $q(x)=p(x|\theta_0)$ and the parameter set is compact. 
Using singular learning theory \cite{Watanabe1,SWatanabeBookE}, it has been proven that
\begin{gather*}
F_n = nS_n + \lambda \log n + O_p(\log \log n), \\
\mathbb{E}[G_n] =\frac{\lambda}{n}+o \left(\frac{1}{n} \right)
\end{gather*}
when $n$ tends to infinity even if the posterior distribution can not be approximated by
any normal distribution, where $S_n$ is the empirical entropy: 
$$S_n = -\frac{1}{n} \sum_{i=1}^n \log q(X_i).$$

The constant $\lambda$ is the RLCT which is
an important birational invariant in algebraic geometry. From a mathematical point of view, the RLCT is
characterized by the following property. We define the zeta function of learning theory by
\begin{equation}\label{zeta}
\zeta(z):=\int \Phi(\theta)^z\varphi(\theta)d\theta,
\end{equation}
where
\begin{equation}\label{ave-err}
\Phi(\theta):=\int q(x)\log\frac{q(x)}{p(x|\theta)}dx.
\end{equation}
$\Phi(\theta)=0$ if and only if $p(x|\theta)=q(x)$ for $x$ almost everywhere.
Let $(-\lambda)$ be the nearest pole of $\zeta(z)$ to the origin; $\lambda$ is then equal to the RLCT. 
If $p(x|\theta)$ is regular, then $\lambda=d/2$. However, this is not true in general. 
The details of the general case are explained in the next section.

\subsection{Relationship between Algebraic Geometry and Learning Theory}
Second, we outline the relationship between algebraic geometry and statistical learning theory. 
Let us describe the motivation behind applying algebraic geometry to learning theory.
As described above, statistical learning encounters a situation in which the true distribution $q(x)$ is not known, although a plurality of data (or sample) $X^n$ can be obtained, where the number of data (or sample size) is $n$. Researchers and practitioners design learning machines or statistical models $p(x|\theta)$ to estimate $q(x)$ by making the predictive distribution $p(x|X^n)$.
At this point, there arises a question, i.e., ``How different is the model from the true distribution?'' 
This issue can be characterized as a model selection problem,  i.e., ``Which model is suitable?'' The ``suitableness'' criteria in this case are the negative log marginal likelihood $F_n$ and the generalization error $G_n$, as mentioned above. However, calculating $F_n$ is very costly for computers, and $G_n$ cannot be computed because $q(x)$ is unknown. Thus, we should estimate them from the data. If the likelihood function $\mathcal{L}(\theta)=\prod_{l=1}^n p(X_l|\theta)$ and the posterior distribution $\psi(\theta|X^n)$ can be approximated by a Gaussian function of $\theta$, we can estimate $F_n$ and $G_n$ by using the Bayesian information criterion (BIC) \cite{Schwarz1978BIC} and Akaike information criterion (AIC) \cite{AkaikeAIC}, respectively. AIC and BIC are respectively defined by
$$AIC = -\frac{1}{n}\sum_{i=1}^n\log p(X_i|\hat{\theta}) + \frac{d}{n}$$
and
$$BIC = -\sum_{i=1}^n \log p(X_i|\hat{\theta}) + \frac{d}{2}\log n,$$
where $\hat{\theta}$ is the maximum likelihood estimator or the maximum posterior estimator and $d$ is the parameter dimension.
AIC and BIC are derived without not using algebraic geometry; however, they are asymptotically equal to $G_n$ and $F_n$ only if $\mathcal{L}(\theta)$ and $\psi(\theta|X^n)$ can approximate a normal distribution. In general, we cannot estimate $G_n$ and $F_n$ by using AIC and BIC; we need algebraic geometry to approximate them.

We describe the framework of analyzing $G_n$ and $F_n$ using algebraic geometry. Consider $\Phi(\theta)$ in Eq. (\ref{ave-err}) and its zero points $\Phi^{-1}(0)$: the zero points of the analytic function form an algebraic variety. We use the following form \cite{Atiyah1970resolution} of the singularities resolution theorem \cite{Hironaka}. This form was originally derived by Atiyah for the analysis of distributions (hyperfunctions); Watanabe later proved that it is useful for creating singular learning theory \cite{Watanabe1}.
\begin{thm}[Singularities Resolution Theorem]\label{sing-resol}
Let $F$ be a non-negative analytic function on the open set $W' \subset \mathbb{R}^d$
and assume that there exists $\theta \in W'$ such that $F(\theta)=0$.
Then, a $d$-dimensional manifold $\mathcal{M}$ and an analytic map $g:\mathcal{M} \rightarrow W'$ exists such that
for each local chart of $\mathcal{M}$, 
\begin{gather*}
F(g(u))=u_1^{2k_1} \ldots u_d^{2k_d}, \\
|g'(u)|=b(u)|u_1^{h_1} \ldots u_d^{h_d}|,
\end{gather*}
where $|g'(u)|$ is the determinant of the Jabobi matrix $g'(u)$ of $g$ and $b:\mathcal{M} \rightarrow \mathbb{R}$ is strictly positive analytic: $b(u)>0$. 
\end{thm}

Thanks to Theorem \ref{sing-resol},
the following analytic theorem has also been proved \cite{Atiyah1970resolution,Bernstein1972,Sato1974zeta}.
\begin{thm}\label{sing-zeta}
Let $F: \mathbb{R}^d \rightarrow \mathbb{R}$ be an analytic function of a variable $\theta \in \mathbb{R}^d$.
Suppose that $a: W \rightarrow \mathbb{R}$ is a $C^{\infty}$-function with compact support $W$. Then,
$$\zeta(z) = \int_W |F(\theta)|^z a(\theta) d\theta$$
is a holomorphic function in $\mathrm{Re}(z)>0$. Moreover, $\zeta(z)$ can be analytically continued to a unique meromorphic
function on the entire complex plane $\mathbb{C}$. The poles of the extended function are all negative rational numbers. 
\end{thm}

The Kullback-Leibler divergence is non-negative; thus, we can apply Theorem \ref{sing-resol} to $\Phi(\theta)$ on $\Phi^{-1}(0) \cap W'$, to get
\begin{gather*}
\Phi(g(u))=u_1^{2k_1} \ldots u_d^{2k_d}, \\
|g'(u)|=b(u)|u_1^{h_1} \ldots u_d^{h_d}|.
\end{gather*}
Assuming the domain of the prior $\varphi(\theta)$ is $W$ and $W \subset W'$, we can also apply Theorem \ref{sing-zeta} to $(\Phi(\theta),\varphi(\theta))$ and obtain Eq. (\ref{zeta}). In this equation, $\zeta(z)$ is called the zeta function of learning theory and it has an analytic continuation on $\mathbb{C}$ that is a unique meromorphic function.
The RLCT of $(\Phi(\theta),\varphi(\theta))$ is defined by the maximum pole of $\zeta(z)$ \cite{SWatanabeBookE}. Furthermore, it has been proved that the RLCT is not dependent on $\varphi(\theta)$ if $0 < \varphi(\theta) < \infty$ on $W$ \cite{SWatanabeBookE}.

Now, let us introduce theorems showing the relationship between RLCTs and $G_n$ and $F_n$
\cite{Watanabe1,SWatanabeBookE,SWatanabeBookMath}.

\begin{thm}\label{watanabe}
Let $q(x)$, $p(x|\theta)$, and $\varphi(\theta)$ be the true distribution, learning machine, and prior distribution,
where $x$ is a point of $\mathbb{R}^N$ and $\theta$ is an element of the compact subset $W$ of $\mathbb{R}^d$.
Put $\Phi(\theta)$ equal to Eq. (\ref{ave-err}) and denote the RLCT of $(\Phi(\theta),\varphi(\theta))$ by $\lambda$.
If there exists at least one $\theta_0$ such that $q(x) = p(x|\theta_0)$,
then the asymptotic behaviors of the generalization error $G_n$ and the free energy $F_n$ are as follows:
\begin{gather*}
\mathbb{E}[G_n] =\frac{\lambda}{n}+o \left(\frac{1}{n} \right), \\
F_n = nS_n + \lambda \log n + O_p(\log \log n).
\end{gather*}
\end{thm}

\begin{thm}\label{watanabe-MAP}
Let $q(x)$, $p(x|\theta)$, and $\varphi(\theta)$ be the true distribution, learning machine, and prior distribution,
where $x$ is a point of $\mathbb{R}^N$ and $\theta$ is an element of the compact subset $W$ of $\mathbb{R}^d$.
Put $\Phi(\theta)$ equal to Eq. (\ref{ave-err}).
If there exists at least one $\theta_0$ such that $q(x) = p(x|\theta_0)$, and the maximum likelihood or posterior method is applied, i.e., the predictive distribution is $p^*(x) = p(x|\hat{\theta})$, where $\hat{\theta}$ is the maximum likelihood or posterior estimator,
then there is a constant $\mu >d/2$ such that the asymptotic behaviors of the generalization error $G_n$ and the free energy $F_n$ are as follows:
\begin{gather*}
\mathbb{E}[G_n] =\frac{\mu}{n}+o \left(\frac{1}{n} \right), \\
F_n = nS_n + \mu \log n + O_p(\log \log n).
\end{gather*} 
\end{thm}

$\Phi(\theta)$ depends on $q(x)$ and $p(x|\theta)$; thus,
Theorem \ref{watanabe} can be understood as meaning that we can determine $G_n$ and $F_n$ if we know the RLCT, which is determined by $(q(x),p(x|\theta),\varphi(\theta))$.
As mentioned above, several studies have sought the RLCT of a statistical model by analyzing the maximum pole of the zeta function. These studies are based on Theorem \ref{watanabe} and the zeta function derived in Theorem \ref{sing-zeta}.
Researchers have found the singularity resolution map $g$ for the exact value or an upper bound of $\Phi(\theta)$
and have obtained the RLCT of the one since the RLCT is order isomorphic: if $\Phi(\theta) \leqq \Psi(\theta)$, then
$\lambda_{\Phi} \leqq \lambda_{\Psi}$, where $(-\lambda_{\Phi})$ and $(-\lambda_{\Psi})$ are the maximum poles of
$\zeta_1(z) = \int \Phi(\theta)^z d\theta$ and $\zeta_2(z) = \int \Psi(\theta)^z d\theta$, respectively \cite{SWatanabeBookE}.

Moreover, from the practical point of view, Theorem \ref{watanabe-MAP} shows that Bayesian inference makes the free energy and the generalization error smaller than those of the maximum likelihood or posterior method in the singular case, since $\mu>d/2 \geqq \lambda$ \cite{SWatanabeBookMath}.
Hence, if the RLCT can be found, we can draw the learning curve as in the right of Fig. \ref{fig:learning-curve} and estimate the sample size with which satisfy the required level of inference performance.

\subsection{Main Theorem}

Now let us introduce the main result of this paper. 
In the following, $\theta=(A,B)$ is a pair of parameter matrices and $x$ is an observed random variable.

A stochastic matrix is defined by a matrix wherein the sum of the elements in a column is equal to 1 and that each entry is non-negative.
For example,
$
\left[ \begin{matrix}
0.1 & 0.1 & 0.4 & 0 \\
0.5 & 0.1 & 0.4 & 0 \\
0.4 & 0.8 & 0.2 & 1
\end{matrix} \right]
$
is a stochastic matrix. It is clear that a product of stochastic matrices is also a stochastic matrix.

Let $K$ be a compact subset of $[0,1]=\{x \in \mathbb{R} | 0 \leqq x \leqq 1\}$ and let 
$K_0$ be a compact of subset of $(0,1)=\{x \in \mathbb{R} | 0<x<1 \}$.
Let $\mathrm{Onehot}(N):=\{w = (w_j) \in \{0,1\}^N \mid \sum_{j=1}^N w_j =1\} = \{ (1,0,\ldots,0),\ldots,(0,\ldots,0,1)\}$ be an $N$-dimensional one-hot vector set and
$\mathrm{Sim}(N,K):=\{c=(c_j) \in K^N \mid \sum_{j=1}^N c_j=1\}$ be a $N$-dimensional simplex.
Let $\mathrm{S}(M,N,E)=\mathrm{Sim}(M,E)^N$ be a set of $M\times N$ stochastic matrices whose elements are in $E$, where $E$ is a subset of $[0,1]$, and $M,N \in \mathbb{N}$. In addition, we set $H,H_0 \in \mathbb{N}$ and $H \geqq H_0$.

In LDA terminology, the number of documents and the vocabulary size is denoted by $N$ and $M$, respectively.
Let $H_0$ be the true or optimal number of topics and $H$ be the chosen one. In this situation, the sample size $n$ is the number of words in all of the given documents. See also Table \ref{params}.

\begin{table}[htb]
  \centering
  \caption{Description of Variables in LDA Terminology}
  \begin{tabular}{|c|c|c|} \hline
    Variable & Description & Index \\
    \hline\hline
    $b_j=(b_{kj}) \in \mathrm{Sim}(H,K)$ & probability that topic is $k$ when document is $j$ & for $k=1,\ldots,H$ \\
    $a_k =(a_{ik}) \in \mathrm{Sim}(M,K)$ & probability that word is $i$ when topic is $k$ & for $i=1,\ldots,M$ \\
    \hline
    $x=(x_i) \in \mathrm{Onehot}(M)$ & word $i$ is defined by $x_i=1$ & for $i=1,\ldots,M$  \\
    $y=(y_k) \in \mathrm{Onehot}(H)$ & topic $k$ is defined by $y_k=1$ & for $k=1,\ldots,H$  \\
    $z=(z_j) \in \mathrm{Onehot}(N)$ & document $j$ is defined by $z_j=1$ & for $j=1,\ldots,N$  \\
    \hline
    $*_0$ and $*^0$ & true or optimal variable corresponding to $*$ & - \\
  \hline
  \end{tabular}
\label{params}
\end{table}


We define $A=(a_{ik}) \! \in \! \mathrm{S}(M,H,K)$ and $B=(b_{kj}) \! \in \! \mathrm{S}(H,N,K)$, and assume that
$A_0=(a^0_{ik}) \in \mathrm{S}(M,H_0,K_0)$ and $B_0=(b^0_{kj}) \in \mathrm{S}(H_0,N,K_0)$ are SMFs such that they give the minimal factorization of $A_0B_0$. We also assume that $\{(a,b,a^0,b^0) \in K^2 \times K_0^2|ab=a^0 b^0 \} \ne \emptyset$.


\begin{defi}[{\bf RLCT of LDA}]
\label{RLCTtopic}
Assume that $M\geqq 2$, $N \geqq 2$, and $H \geqq H_0 \geqq 1$.
Let $q(x|z)$ and $p(x|z,A,B)$ be conditional probability density functions of $x \in \mathrm{Onehot}(N)$ given $z \in \mathrm{Onehot}(M)$, which represent 
the true distribution and the learning machine, respectively,
\begin{align}
\label{topic-true}
q(x|z) & = \prod_{j=1}^N \left(\sum_{k=1}^{H_0} b^0_{kj} \prod_{i=1}^M (a^0_{ik})^{x_{i}}\right)^{z_j}, \\
\label{topic-model}
p(x|z,A,B) & = \prod_{j=1}^N \left(\sum_{k=1}^{H} b_{kj} \prod_{i=1}^M (a_{ik})^{x_{i}}\right)^{z_j}. 
\end{align}
These distributions are the marginalized ones of the following simultaneous ones with respect to the topics $y^0 \in \mathrm{Onehot}(H_0)$ and $y \in \mathrm{Onehot}(H)$:
\begin{align*}
q(x,y^0|z) & = \prod_{j=1}^N \left[\prod_{k=1}^{H_0} \left(b^0_{kj} \prod_{i=1}^M (a^0_{ik})^{x_{i}}\right)^{y^0_k}\right]^{z_j}, \\
p(x,y|z,A,B) & = \prod_{j=1}^N \left[\prod_{k=1}^{H} \left(b_{kj} \prod_{i=1}^M (a_{ik})^{x_{i}}\right)^{y_k}\right]^{z_j}. 
\end{align*}
In practical cases, the topics are not observed; thus, we use Eq. (\ref{topic-true}) and (\ref{topic-model}).


In addition, let $\varphi(A,B) >0$ be a probability density function such that it is positive on a compact subset of 
$\mathrm{S}(M,H,K) \times \mathrm{S}(H,N,K)$ including $\Phi^{-1}(0)$ i.e. $(A_0,B_0)$.
Put
$$\mathrm{KL}(A,B):=\sum_{z \in \mathrm{Onehot}(M)} \sum_{x \in \mathrm{Onehot}(N)} q(x|z)q'(z) \log \frac{q(x|z)}{p(x|z,A,B)},$$
where $q'(z)$ is the true distribution of the document. In LDA, $q'(z)$ is not observed and assumed that it is positive and bounded.

Then, the holomorphic function of one complex variable $z \ (\mathrm{Re} (z) >0)$
\[
\zeta(z)=\int_{\mathrm{S}(M,H,K) } dA \int_{\mathrm{S}(H,N,K) } dB \
 \mathrm{KL}(A,B)^z 
\]
can be analytically continued to a unique meromorphic function on the entire complex plane $\mathbb{C}$ and all of its poles are rational and negative. If the largest pole is $(-\lambda)$, then $\lambda$ is said to be the RLCT of LDA.
\end{defi}

\begin{defi}[{\bf RLCT of SMF}]
\label{RLCTSMF}
Set $\Phi(A,B)=\|AB-A_0 B_0\|^2$.
Then the holomorphic function of one complex variable $z \ (\mathrm{Re} (z) >0)$
\[
\zeta(z)=\int_{\mathrm{S}(M,H,K) } dA \int_{\mathrm{S}(H,N,K) } dB \
 \Phi(A,B)^z 
\]
can also be analytically continued to a unique meromorphic function on $\mathbb{C}$ and its all poles are rational and negative. If the largest pole is $z=-\lambda$, then $\lambda$ is the RLCT of SMF.
\end{defi}

In this paper, we prove the following two theorems.

\begin{thm}[{\bf Equivalence of the LDA and SMF}]
\label{topicSMF} Let $\lambda_{SMF}$ be the RLCT of SMF and $\lambda_{LDA}$ be the RLCT of LDA.
Then, $\lambda_{SMF} = \lambda_{LDA}$.
\end{thm}

\begin{thm}[{\bf Main Theorem}]\label{thm:main}
If
$M\geqq 2$, $N \geqq 2$, and $H \geqq H_0 \geqq 1$,
then the RLCT of LDA $\lambda$ satisfies the following inequality:
\begin{equation}\label{main-ineq}
\lambda \leqq \frac{1}{2}\left[
M-1+(H_0-1)(M+N-3)+(H-H_0)\min\{M-1,N\}
\right].
\end{equation}
In particular, equality holds if $H=H_0=1$ or $H=H_0=2$:
$$\lambda =\begin{cases}
\frac{M-1}{2} & (H=H_0=1) \\
\frac{2M+N-4}{2} & (H=H_0=2)
\end{cases}.$$
Also, if $H=2$ and $H_0=1$, then
$$\lambda =\begin{cases}
M-1 & (M \geqq N) \\
\frac{M+N-2}{2} & (M<N)
\end{cases}.$$
\end{thm}

We prove Theorem \ref{topicSMF} and \ref{thm:main} in the third and fourth sections.
As applications of them, we obtain an upper bound of the free energy and Bayesian generalization error in LDA and SMF. 

\begin{thm}\label{thm:topic}
Under the same assumptions as Definition \ref{RLCTtopic},
the negative log marginal likelihood (free energy) $F_n$ and the expected generalization error $\mathbb{E}[G_n]$ in LDA
satisfy the following inequalities as $n \rightarrow \infty$:
\begin{gather*}
F_n \leqq nS_n {+} \overline{\lambda} \log n {+} O_p(\log \log n), \\
\mathbb{E}[G_n] \leqq \frac{\overline{\lambda}}{n}+o\left(\frac{1}{n}\right),
\end{gather*}
where $\overline{\lambda}$ is the upper bound of the RLCT of LDA in Theorem \ref{thm:main}.
\end{thm}

Here, we will research the case that a set of words in all the documents is $\{x(1),\ldots,x(n)\}$,
where $x(l)$ is the $l$-th word. For word $x(l)$, let $y(l)$ and $z(l)$ be the corresponding topic and document, respectively.
Then, the likelihood is given by
$$\mathcal{L}(A,B)=\prod_{l=1}^n p(x(l)|z(l),A,B).$$
Thus, the posterior can be defined by the normalizing of the product of the above likelihood and prior:
$$\psi(A,B|x(1),\ldots,x(n)) = \frac{\prod_{l=1}^n p(x(l)|z(l),A,B) \varphi(A,B)}{\iint dAdB \prod_{l=1}^n p(x(l)|z(l),A,B) \varphi(A,B)}.$$

This theoretical result leads us to the following theorem.

\begin{thm}\label{thm:SMF}
Assume that $M\geqq 2$, $N \geqq 2$, and $H \geqq H_0 \geqq 1$ and $X$ is an observed random matrix.
Let $q(X)$ and $p(X|A,B)$ be probability density functions of $X\in \mathrm{S}(M,N,K)$ that represent 
the true distribution and learning machine, respectively;
\begin{align*}
q(X) \propto \exp \left( -\frac{1}{2}\|X-A_0 B_0\|^2 \right), \quad
p(X|A,B) \propto \exp \left(-\frac{1}{2}\|X-AB\|^2 \right). 
\end{align*}
In addition, let $\varphi(A,B) >0$ be a probability density function such that it is positive on a compact subset of 
$\mathrm{S}(M,H,K) \times \mathrm{S}(H,N,K)$ including $\Phi^{-1}(0)$ i.e. $(A_0,B_0)$.
Then, $\Phi(A,B)$ has the same RLCT as $\|AB-A_0 B_0\|^2$ and the free energy $F_n$ and the expected generalization error $\mathbb{E}[G_n]$ behave as in Theorem \ref{thm:topic} for $n \rightarrow \infty$.
\end{thm}

Regarding this theorem, we will study the case in which a number of random matrices $\{X_1,\ldots,X_n\}$ 
are observed and the true decomposition $A_0$ and $B_0$ is statistically estimated. A statistical model
$p(X|A,B)$ with parameters $(A,B)$ is used for inference. Thus, the theorem gives the theoretical Bayesian generalization error. 
Indeed,  as described in Section \ref{discuss}, Theorem \ref{thm:topic} also applies when $q(X)$ and $p(X|A,B)$ are Poisson, exponential, or Bernoulli distributions.

Theorem \ref{thm:topic} and \ref{thm:SMF} immediately follow from Theorem \ref{topicSMF} and \ref{thm:main}, which are proved subsequently.

\section{Preparations}
Let $A \! \in \! \mathrm{S}(M,H,K)$ and $B \! \in \! \mathrm{S}(H,N,K)$ be
$$A\! = \!(a_1,\ldots,a_H) , a_{k}\!=\!(a_{ik})_{i=1}^{M},$$ 
$$B\! = \!(b_1,\ldots,b_H)^T , b_{k}\!=\!(b_{kj})_{j=1}^{N},$$
and
$A_0 \! \in \! \mathrm{S}(M,H_0,K_0)$ and $B_0 \! \in \! \mathrm{S}(H_0,N,K_0)$ be
$$A_0\! =\! (a^0_1,\ldots,a^0_{H_0}) , a^0_{k}\!=\!(a^0_{ik})_{i=1}^{M},$$
$$B_0\! =\! (b^0_1,\ldots,b^0_{H_0})^T , b^0_{k}\!=\!(b^0_{kj})_{j=1}^{N}.$$
$A, B, A_0$, and $B_0$ are stochastic matrices; thus,
$$a_{Mk}=1-\sum_{i=1}^{M-1} a_{ik}, \quad b_{Hj}=1-\sum_{k-1}^{H-1} b_{kj},$$
$$a^0_{Mk}=1-\sum_{i=1}^{M-1} a^0_{ik}, \quad b^0_{H_0 j}=1-\sum_{k-1}^{H_0-1} b^0_{kj}.$$

We need the following four lemmas and two propositions in order to prove the Main Theorem. 

Let $F$ and $G$ be non-negative analytic functions from a subset $W$ of Euclidian space to $\mathbb{R}$. The RLCT of $F$ is defined by $\lambda$, where $(-\lambda)$ is the largest pole of the following function:
$$\zeta(z)=\int_W dw F(w)^z$$
which is analytically connected to the entire complex plane as a unique meromorphic function.
When the RLCT of $F$ is equal to the RLCT of $G$, we denote this situation by $F \sim G$. Regarding the binomial relation $\sim$, the following propositions are known.

\begin{prop}
\label{idealRLCT}
Suppose $s, t\in \mathbb{N}$, and let $f_1(w),\ldots,f_s(w),g_1(w),\ldots,g_t(w)$ be real polynomials. Furthermore, let 
\begin{align*}
I:=\langle f_1,\ldots,f_s \rangle, \quad
J:=\langle g_1,\ldots,g_t \rangle
\end{align*}
be the generated ideal of $(f_1,\ldots,f_s)$ and $(g_1,\ldots,g_t)$, respectively.
We put
\begin{align*}
F(w):=\sum_{i=1}^s f_i(w)^2, \quad
G(w):=\sum_{j=1}^t g_j(w)^2.
\end{align*}
Then, $I=J$ if and only if $F \sim G$.
\end{prop}

\begin{proof}
This proposition follows immediately proved from the Cauchy-Schwarz' inequality. \qed
\end{proof}

The above leads to the following corollary.

\begin{cor}
\label{idealRLCT2}
Assume that $F(w)=\sum_{i=1}^s f_i(w)^2$. Then
$$F(w) + \left( \sum_{i=1}^s f_i(w) \right)^2\sim F(w).$$
\end{cor}

\begin{proof}
We can easily prove this by using $\sum_{i=1}^s f_i(w) \in I$ and Proposition \ref{idealRLCT}. \qed
\end{proof}

\begin{prop}
\label{idealNMF}
Put $(x_i)_{i=1}^M ,(a_i)_{i=1}^M \in K_1 \subset \mathbb{R}^M$ and $(y_j)_{j=1}^N, (b_j)_{j=1}^N \in K_2 \subset \mathbb{R}^N$, where $K_1$ and $K_2$ are compact sets that do not include 0. Let $f_{ij}$ be $x_i y_j-a_ib_j$ and $I$ be $\left\langle (f_{ij})_{(i,j)=(1,1)}^{(M,N)} \right\rangle$. Then,
$$I=\langle f_{11}, f_{21},\ldots,f_{M1},f_{12},\ldots,f_{1N} \rangle$$
and
$$\sum_{i=1}^M \sum_{j=1}^N f_{ij}^2 \sim \sum_{i=2}^M f_{i1}^2 + \sum_{j=2}^N f_{1j}^2 +f_{11}^2.$$
\end{prop}


We rigorously proved Proposition \ref{idealNMF} in our previous research (Lemma 3 and 4 in \cite{nhayashi2}).
In addition, it is easily verified that the RLCT $\lambda$ of $\Phi(x,y)=\sum_{i=1}^M \sum_{j=1}^N f_{ij}^2 = \sum_{i=1}^M \sum_{j=1}^N (x_i y_j -a_i b_j)^2$ equals
$\lambda = (M+N-1)/2$ by using blowing-up 
and Proposition \ref{idealNMF}.

The above propositions enable us to prove the following four lemmas.
The rigorous proof is in Appendix \ref{pf-lem}.

\begin{lem}
\label{lemH1}
If $H=H_0=1$ (the stochastic matrix $B$ is constant), $\lambda =(M-1)/2$.
\end{lem}

\begin{lem}
\label{lemH21}
Let $\lambda$ be the RLCT of $\|AB-A_0 B_0\|^2$. If $M \geqq 2$,$N \geqq 2$, $H=2$, and $H_0=1$,
$$\lambda =\begin{cases}
M-1 & (M \geqq N) \\
\frac{M+N-2}{2} & (M<N)
\end{cases}.$$
\end{lem}

\begin{lem}
\label{lemH22}
If $M\geqq 2$, $N \geqq 2$, and $H=H_0=2$, then the Main Theorem holds with equality:
$$\lambda = \frac{2M+N-4}{2}.$$
\end{lem}

\begin{lem}
\label{lemH}
Suppose $M \geqq 2$ and $N \geqq 2$. In the case of $H=H_0$, the Main Theorem holds:
$$\lambda \leqq \frac{1}{2}\{ M-1 + (H-1)(M+N-3)\}.$$
\end{lem}

%
\section{Proof of the Main Theorem}

We will prove Theorem \ref{topicSMF} and then use it to prove the Main Theorem \ref{thm:main}.
In other words, we use Theorem \ref{topicSMF} to relate the SMF to LDA.

First, we prove the equivalence of LDA and SMF.
\begin{proof}[Proof of Theorem \ref{topicSMF}]
Without loss of generality, we can rewrite the notation of $q(x|d)$ and $p(x|d,A,B)$ as follows:
\begin{align*}
q(x|z_i=1) = \sum_{k=1}^{H_0} b^0_{kj} \prod_{i=1}^M (a^0_{ik})^{x_{i}}, \quad
p(x|z_i=1,A,B) = \sum_{k=1}^{H} b_{kj} \prod_{i=1}^M (a_{ik})^{x_{i}}. 
\end{align*}
The word $x$ is a one-hot vector; hence, we obtain
\begin{align*}
q(x_j=1|z_i=1) = \sum_{k=1}^{H_0} a^0_{ik} b^0_{kj}, \quad
p(x_j=1|z_i=1,A,B) = \sum_{k=1}^{H} a_{ik} b_{kj}. 
\end{align*}

Then, the conditional Kullback-Leibler divergence between $q(x|d)$ and $p(x|d,A,B)$ is equal to
\begin{align*}
\mathrm{KL}(A,B) &= \sum_{z \in \mathrm{Onehot}(M)} \sum_{x \in \mathrm{Onehot}(N)} q(x|z)q'(z) \log \frac{q(x|z)}{p(x|z,A,B)} \\
&=\sum_{i=1}^M \sum_{j=1}^N q(x_j=1|z_i=1)q'(z_i=1) \log \frac{q(x_j=1|z_i=1)}{p(x_j=1|z_i=1,A,B)} \\
&=\sum_{j=1}^N q'(z_i=1) \sum_{i=1}^M \left( \sum_{k=1}^{H_0} a^0_{ik} b^0_{kj} \right) \log \frac{\sum_{k=1}^{H_0} a^0_{ik} b^0_{kj}}{\sum_{k=1}^{H} a_{ik} b_{kj}}.
\end{align*}

Owing to $A=(a_{ik}) \! \in \! \mathrm{S}(M,H,K)$, $B=(b_{kj}) \! \in \! \mathrm{S}(H,N,K)$,
$A_0=(a^0_{ik}) \in \mathrm{S}(M,H_0,K_0)$, and $B_0=(b^0_{kj}) \in \mathrm{S}(H_0,N,K_0)$,
the $(i,j)$ entries of $AB$ and $A_0B_0$ are $(AB)_{ij}:=\sum_{k=1}^{H} a_{ik} b_{kj}$ and $(A_0B_0)_{ij}:=\sum_{k=1}^{H_0} a^0_{ik} b^0_{kj}$.
We have
\begin{align}
\label{KL-TM}
\mathrm{KL}(A,B)
&=\sum_{j=1}^N q'(z_i=1) \sum_{i=1}^M (A_0B_0)_{ij} \log \frac{(A_0B_0)_{ij}}{(AB)_{ij}}.
\end{align}

According to \cite{Matsuda1-e}, $\sum_{i=1}^M (A_0B_0)_{ij} \log \frac{(A_0B_0)_{ij}}{(AB)_{ij}}$ in Eq. (\ref{KL-TM}) has the same RLCT of
$\sum_{i=1}^M ((AB)_{ij} - (A_0B_0)_{ij} )^2$.
In addition, $q'(z_i=1)$ is positive and bounded. Accordingly, we have
\begin{align}
\label{TM=SMF}
\mathrm{KL}(A,B)
&=\sum_{j=1}^N q'(z_i=1) \sum_{i=1}^M (A_0B_0)_{ij} \log \frac{(A_0B_0)_{ij}}{(AB)_{ij}} \\
&\sim \sum_{j=1}^N q'(z_i=1) \sum_{i=1}^M ((AB)_{ij} - (A_0B_0)_{ij} )^2 \\
&\sim \sum_{j=1}^N \sum_{i=1}^M ((AB)_{ij} - (A_0B_0)_{ij} )^2 = \|AB-A_0B_0\|^2.
\end{align}

Therefore, $\mathrm{KL}(A,B) \sim \|AB-A_0B_0\|^2$; i.e., the RLCT of LDA equals the RLCT of SMF.
\qed
\end{proof}

Here, we sketch the proof of the Main Theorem and gives two remarks on it.
The rigorous proof is in Appendix \ref{pf-main}. 
\begin{proof}[Sketch of Proof of Theorem \ref{thm:main}]

Because of Theorem \ref{topicSMF}, we have only to prove for the RLCT of SMF, i.e. the zero points of $\|AB-A_0B_0\|^2$.

First, we express $\Phi(A,B)=\|AB-A_0B_0\|^2$ in terms of its components and have
\begin{align*}
\Phi(\!A,B\!)&= \sum_{j=1}^N \sum_{i=1}^{M-1} \left\{ \!\sum_{k=1}^{H_0-1} (\!a_{ik}b_{kj} {-}a^0_{ik}b^0_{kj})+a_{iH_0}b_{H_0j}-a^0_{iH_0}b^0_{H_0j}
+\sum_{k=H_0+1}^{H-1}a_{ik}b_{kj} + a_{iH} b_{Hj} \!\right\}^2 \nonumber \\
&\quad + \sum_{j=1}^N \left\{ \!\sum_{k=1}^{H_0-1} (\!a_{Mk}b_{kj} {-}a^0_{Mk}b^0_{kj}){+}a_{MH_0}b_{H_0j}{-}a^0_{MH_0}b^0_{H_0j}
{+}\sum_{k=H_0+1}^{H-1}a_{Mk}b_{kj} {+} a_{MH} b_{Hj}\! \right\}^2\!\!.
\end{align*}
Thus,
\begin{align*}
\Phi(\!A,B\!)
&= \sum_{j=1}^N \sum_{i=1}^{M-1} K_{ij}^2 +\sum_{j=1}^N L_j^2 = \sum_{j=1}^N \sum_{i=1}^{M-1} K_{ij}^2 +\sum_{j=1}^N \left( \sum_{i=1}^{M-1} K_{ij} \right)^2. 
\end{align*}
where
\begin{gather*}
K_{ij}:=\sum_{k=1}^{H_0-1} (a_{ik}b_{kj} -a^0_{ik}b^0_{kj})+a_{iH_0}b_{H_0j}-a^0_{iH_0}b^0_{H_0j}
+\sum_{k=H_0+1}^{H-1}a_{ik}b_{kj} + a_{iH} b_{Hj} , \\
L_j:= \sum_{k=1}^{H_0-1} (a_{Mk}b_{kj} -a^0_{Mk}b^0_{kj})+a_{MH_0}b_{H_0j}-a^0_{MH_0}b^0_{H_0j}
+\sum_{k=H_0+1}^{H-1}a_{Mk}b_{kj} + a_{MH} b_{Hj}.
\end{gather*}
Thanks to Corollary \ref{idealRLCT2}, we obtain
$\Phi(\!A,B\!) \sim \sum_{j=1}^N \sum_{i=1}^{M-1} K_{ij}^2.$

Second, we calculate the RLCTs of the terms of the bound. Using linear transformations and the triangle inequality, 
\begin{align*}
\Phi(A,B) \leqq C_1 \sum_{j=1}^N \sum_{i=1}^{M-1}  \left\{ \sum_{k=1}^{H_0-1} (a_{ik}b_{kj}-a^0_{ik}b^0_{kj}) + c_i \right\}^2 
{+} C_2 \sum_{k=H_0}^{H-1} \left( \sum_{j=1}^N b_{kj}^2 \right) \left( \sum_{i=1}^{M-1} a_{ik}^2 \right),
\end{align*}
for some constants $C_1>0$ and $C_2>0$. 
Therefore, by making blow-ups of the respective variables $\{ a_{ik} \}$ and $\{ b_{kj} \}$ and applying Lemma \ref{lemH},
we arrive at
$$\lambda \leqq \frac{1}{2}\left[
M-1+(H_0-1)(M+N-3)+(H-H_0)\min\{M-1,N\}
\right].$$ \qed
\end{proof}

\begin{rem}\label{remEq}
The equality in the Main Theorem holds if $H=H_0=1$ or $H=H_0=2$.
If $H=2$ and $H_0=1$, the bound in the Main Theorem is not equal to the exact value of $\lambda$.
\end{rem}
\begin{proof}
The Main Theorem, Lemma \ref{lemH1}, \ref{lemH21}, and \ref{lemH22} immediately lead to the statement. \qed
\end{proof}

\begin{rem}\label{remEqH}
Under the same assumptions as in the Main Theorem, suppose 
$$\forall k \ (H_0 \leqq k \leqq H-1), \ a_{ik} \geqq a_{iH}.$$
Then, $\bar{\lambda}_2$ used in the proof of Main Theorem in Appendix \ref{pf-main}.
is equal to
$$\bar{\lambda}_2 =\frac{(H-H_0)\min\{M-1,N\}}{2}$$
and the upper bound of $\lambda$ in the Main Theorem becomes tighter than that under the original assumptions.
\end{rem}

\begin{proof}
In the same way as in the proof of the Main Theorem or Lemma 3.1 in our previous result \cite{nhayashi2}, 
the claim is attained.
\qed
\end{proof}

\section{Discussion}\label{discuss}

We shall discuss the results of this paper from three viewpoints. 

\subsection{Tightness of the Upper Bound}
First, let us consider the tightness of the upper bound.
In general, if a prior is not zero or infinity in a neighborhood of $\theta_0$, then the RLCT is bounded by $d/2$ \cite{SWatanabeBookE},
where $\theta_0$ is the true parameter and $d$ is the dimension of the parameter space. 
The dimension of the parameter space in SMF is equal to the number of elements in the learner matrices $H(M+N)$; however, the learner matrices $A$ and $B$ are stochastic; thus, they have $H+N$ degrees of freedom. Hence, the essential dimension $d$ equals
$H(M+N)-H-N$.
Let $\bar{\lambda}$ be the upper bound described in the Main Theorem.
We can immediately verify that the bound is non-trivial, i.e., $\bar{\lambda} < d/2$.

It is supposed that this result is due to the lowness of the exact values determined in this paper.
In general, if the learning machine exactly matches the true distribution, then the RLCT of the model is equal to half of the dimension, i.e. $d/2$ in consideration of the degrees of freedom \cite{SWatanabeBookE}. For instance, in reduced rank regression, i.e. in conventional matrix factorization in which the elements of the matrices are in $\mathbb{R}$, if the learner rank $H$ is equal to the true rank, then the RLCT $\lambda_R$ is equal to
$$\lambda_R=\frac{H(M+N)-H^2}{2}=\frac{H(M+N-H)}{2}=\frac{d}{2},$$
where $M$ and $N$ are the input and output sizes, respectively \cite{Aoyagi1}. This means that the exact degree of freedom is equal to $H^2$. This is because the learner matrix $AB$ equals $AP^{-1}PB$, where $P$ is an $H \times H$ regular matrix; i.e., $P$ is an element of a general linear group $\mathrm{GL}(H,\mathbb{R})$ whose dimension is $H^2$.
However, in SMF, the exact value $\lambda$ of the RLCT does not equal $d/2$ when $H=H_0=2$:
$$\lambda=\frac{2M+N-4}{2}<\frac{2(M+N)-2-N}{2}=\frac{2M+N-2}{2}=\frac{d}{2}.$$
Also, $\lambda$ is less than the exact value of the RLCT of reduced rank regression:
$$\lambda=\frac{2M+N-4}{2}<\frac{2(M+N)-4}{2}=\lambda_R.$$
We hence conclude that the degree of freedom $r$ in SMF is {\it not} equal to $H+N$ or $H^2$ and it satisfies $H+N<r$.
For example, if $H=H_0=2$, $$r=2(M+N)-2\lambda=2M+2N-2M-N+4=N+4>N+2=N+H.$$
This difference occurs because of the stochastic condition: the entries of matrices are in $[0,1]$ and the sum of the elements in a column is equal to 1. In general, this condition directly has $r \geqq H+N$; however, from an indirect point of view, the dimension of the space of $\{ P \in \mathrm{GL}(H,\mathbb{R}) \mid AB=AP^{-1}PB \ {\rm and} \ AP^{-1} \ {\rm and} \ PB \ {\rm are stochastic} \}$ is not clear.
This difficulty also appears in NMF because the usual rank does not equal the non-negative rank \cite{nhayashi5}. Numerical experimental analysis indicates that the RLCT of NMF may be larger than $d/2$ of reduced rank regression even when $H=H_0$ \cite{nhayashi5}.
We presume that there exists a special rank that is defined by the minimal $H_0$ in SMF that may be called ``stochastic rank''.
The above problems give us considerable prospects for future research.

\subsection{Robustness of the Result for Other Distributions} 
Second, let us consider generalizing our result to another distribution.
In Theorem \ref{thm:topic}, we considered the case in which the matrix $X$ is subject to a normal distribution whose averages are $A_0 B_0$ and $AB$. Then, the Kullback-Leibler divergence (KL-div) $\mathrm{KL}(A,B)$ of the true distribution and the learning machine satisfies $\mathrm{KL}(A,B) \sim \|AB-A_0B_0\|^2=\Phi(A,B)$, as is well known \cite{Aoyagi1}. If $X$ is subject to a Poisson distribution or an exponential distribution, the elements of $AB$ must be restricted by strictly positive elements. However, it has been proved that the KL-div has the same RLCT as the square error $\Phi(A,B)$ when the elements of $AB$ are strictly positive \cite{nhayashi2}.

Let us also study the case in which $X$ is subject to a Bernoulli distribution when the elements of $AB$ are strictly positive and less than one. In particular, we will consider Bernoulli distributions whose averages are the elements of $A_0 B_0$ and $AB$.
This means that the sample is a set of binary matrices; this sort of problem appears in consumer analysis and text modeling \cite{LarsenNMF4Binary}. Binary data are frequently generated in text analysis, sensory data, and market basket data. From a statistical point of view, it can be understood that binary matrices are subject to a Bernoulli distribution whose average is represented by a stochastic matrix $C \in \mathrm{S}(M,N,K)$. We treat this average matrix, i.e., the parameter matrix of the Bernoulli distribution, and factorize it as ``$C = A_0B_0 \approx AB$''. The double quotation marks mean that this equality is in a statistical sense, not a deterministic one. According to \cite{LarsenNMF4Binary}, NMF for binary matrix data is useful in the fields mentioned above and Section \ref{sec-intro}.
In order to apply the Main Theorem to this problem, we need to prove the following proposition.
\begin{prop}\label{bayes-binary}
Let $q(X)$ and $p(X|A,B)$ be probability density functions of an $M \times N$ binary matrix $X$ that respectively represent the true distribution and the learning machine,
\begin{gather*}
q(X) \propto \mathrm{Ber}(X|A_0B_0), \quad
p(X|A,B) \propto \mathrm{Ber}(X|AB),
\end{gather*}
where $\mathrm{Ber}(X|C)$ is a probability density function of a Bernoulli distribution with average $C$.
Also, let $\varphi(A,B)$ be a probability density function that is bounded and positive on a compact subset of $\mathrm{S}(M,H,K)\times \mathrm{S}(H,N,K)$ including $(A_0,B_0)$.
Then, the KL-div of  $q(X)$ and $p(X|A,B)$ has the same RLCT as the square error $\Phi(A,B)$.
\end{prop}

Because of Proposition \ref{bayes-binary}, the Main Theorem gives an upper bound of the expected Bayesian generalization error in an NMF for binary data.
The proof of Proposition \ref{bayes-binary} is below.
\begin{proof}
Let $x \in \{0,1\}$, $0<a<1$, and $0<b<1$. 
We put 
\begin{align*}
  p(x|a) &:= a^x (1-a)^{1-x}, \\
  \Phi(a,b) &:= \sum^{1}_{x=0}p(x|a) \mathrm{log} \frac{p(x|a)}{p(x|b)}.
\end{align*}
For simplicity, we write $\sum$ instead of $\sum_{x=0}^1$.
Using
\begin{align*}
  \log \frac{p(x|a)}{p(x|b)}
  &= \log a^x (1-a)^{1-x} - \log b^x(1-b)^{1-x} \\
  &= x(\log a-\log b) + (1-x) \{ \log (1-a) -\log (1-b) \},
\end{align*}
$\sum p(x|a)=1,$ and $\sum xp(x|a)=a,$
we have
\begin{align*}
  \Phi(a,b) &= \sum x p(x|a) (\log a-\log b) +\sum (1-x)p(x|a) \{ \log (1-a) -\log (1-b) \} \\
  &= a(\log a- \log b) + (1-a)\{ \log (1-a) - \log (1-b) \}.
\end{align*}
To simplify the notation, we will use the abbreviated symbol for the partial derivative, i.e., $\partial_{\theta}$ instead of $\partial/\partial \theta$.
Then, owing to
\begin{align*}
\partial_a \Phi(a,b) &= \log\frac{a}{b}-\log \frac{1-a}{1-b}, \\
\partial_b \Phi(a,b) &= \frac{1-a}{1-b}-\frac{a}{b},
\end{align*}
and that the log function is monotone increasing, we have
$$\partial_a \Phi(a,b)=\partial_b \Phi(a,b)=0 \Leftrightarrow a=b.$$
The signs of the above partial derivatives are
$$\begin{cases}
\partial_a \Phi(a,b)>0 \wedge \partial_b \Phi(a,b)<0 & a>b,\\
\partial_a \Phi(a,b)<0 \wedge \partial_b \Phi(a,b)>0 & a<b.
\end{cases}$$

On account of the signs of the derivatives and smoothness, the increase in (or decrease in) and convexity of $\Phi(a,b)$ are the same as those of $(b-a)^2$.
Hence, $\exists c_1 , c_2>0$ s.t. 
\begin{eqnarray}
\label{KL2RSS}
c_1 (b-a)^2 \leqq \Phi(a,b) \leqq c_2 (b-a)^2,
\end{eqnarray}
i.e. $\Phi(a,b) \sim (b-a)^2$.

Let us assume that matrix elements are generated from Bernoulli distributions. Using inequality (\ref{KL2RSS}) for each element,
we have
$$\mbox{KL-div} \sim \|AB-A_0 B_0 \|^2,$$
where $a$ is an element of $A_0 B_0$ and $b$ is an element of $AB$.
\qed
\end{proof}

\begin{rem}
If $A$, $B$, $A_0$, and $B_0$ are not stochastic but their elements are in $(0,1)$, then the RLCT is equal to the RLCT of NMF and the Bayesian generalization error can be bounded by our previous result {\rm \cite{nhayashi2,nhayashi5}}, since the above proof can be used for applying to the bound of the RLCT of NMF.
\end{rem}

\subsection{Application to Markov Chain}\label{app2BN}


Third, 
let us study an application of the main result. Markov chains and Bayesian networks are used for many purposes, such 
marketing \cite{styan1964markov,roje2017consumption}, 
and weather forecasting \cite{caskey1963markov,sonnadara2015markov}.
Here, SMF can be used in the inference of a Bayesian network composed of a Markov chain; this Bayesian network is one of the simplest and non-trivial ones. That is, it has been shown that reduced rank regression for a Markov chain $y=Cx$ is statistically equivalent to applying SMF: $AB \approx A_0B_0 =C$ if $A$, $B$, $A_0$, and $B_0$ are stochastic matrices (see also \cite{Aoyagi1}).


Suppose we want to estimate a linear map whose transition stochastic matrix $C$ has a lower rank $H_0$ than the dimension of the given input $N$ and output $M$. Here, the stochastic matrix $C$ can be decomposed into $A_0$ and $B_0$ whose ranks are $H_0$, but we do not know $H_0$ or $(A_0,B_0)$. These unknowns are called the true rank and the true parameter, respectively. The Main Theorem can be applied to this problem.
\begin{prop}\label{bayes-net}
Let $q(x)$ be the true distribution of the input such that an $N \times N$ matrix $\mathcal{X}:=(\int x_i x_j q(x)dx)$ is positive definite, where $x = (x_1,\ldots,x_N)$. Let $q(y|x)$ and $p(y|x;A,B)$ be conditional probability density functions of the output $y \in \mathbb{R}^M$ given the input $x \in \mathbb{R}^N$ that respectively represent the true distribution and the learning machine:
\begin{gather*}
q(y|x) \propto \exp\left( -\frac{1}{2} \|y-A_0B_0x \|^2 \right), \quad
p(y|x;A,B) \propto \exp \left( -\frac{1}{2} \|y-ABx\|^2 \right).
\end{gather*}
In addition, let $\varphi(A,B)$ be a probability density function that is bounded and positive on a compact subset of $\mathrm{S}(M,H,K)\times \mathrm{S}(H,N,K)$ including $(A_0,B_0)$.
Then, the KL-div of $q(y|x)$ and $p(y|x;A,B)$ has the same RLCT as the square error between the product of the learner matrices $AB$ and one of the true parameters $A_0B_0$.
\end{prop}

This problem is similar to reduced rank regression in which the representation matrix of the linear map is not restricted to being stochastic; the elements are just real numbers, however, this stochastic condition makes the exact value of the RLCT unclear. Thus, we are only able to give an upper bound of the RLCT. 

Proposition \ref{bayes-net} immediately follows from Lemma 1 in \cite{Aoyagi1}.
Therefore, the Main Theorem gives an upper bound of the expected Bayesian generalization error in the above type of Markov chain.

\subsection{Novelty of Proof and Method to find RLCT}
Lastly, let us discuss the novelty of our proof. As mentioned in Section \ref{sec-intro},
there are different methods for finding RLCTs for learning machines. These methods are based on the theory of the zeta function such as in Theorem \ref{sing-zeta}, and researchers sometimes use blow-ups of the parameter variables.
However, there is no standard method to analytically compute RLCTs of collections of functions. Since learning machines are different functions depending on control variables, they form families of functions. Control variables also depend on the actual learning machine, for example, the number of topics in the topic model, the inner dimension of the product of learner matrices for MF, NMF, and SMF, the number of components in the mixture models, the number of hidden units in neural networks and reduced rank regression, etc.

For instance, in the Gaussian mixture model, each density function $p$ of the model is different from the number of components $K$:
\begin{align*}
p(x|a,\mu) &= \mathcal{N}(x|\mu,1), \ \mbox{if} \ K=1, \\
p(x|a,\mu) &= a\mathcal{N}(x|\mu_1,1) + (1-a)\mathcal{N}(x|\mu_2,1), \ \mbox{if} \ K=2,
\end{align*}
where $a$ is the mixing ratio, $\mu$ is the center of the each component, and $\mathcal{N}(x|m,s)$ is the density function of a normal distribution whose average and standard derivation are $m$ and $s$ respectively.
Yamazaki derived an upper bound of the RLCT of the Gaussian mixture model \cite{Yamazaki1}.

In our problem, $H$ is the control variable. It is true that we can explain this problem as a singularity resolution of $\mathbb{V}(A,B):=\{A \in \mathrm{S}(M,H,K),B \in \mathrm{S}(H,N,K) \mid \|AB-A_0B_0\|^2=0\}$, but $\mathbb{V}(A,B)$ is different from each $H$; hence
we must consider a family of functions. In fact, we proved four lemmas that give the exact value or the upper bound for each case. We merged them and derived a general upper bound (the Main Theorem). This paper gives a general solution for the RLCT of the topic model and SMF.

%

\section{Conclusion}
By using an algebraic and geometrical method, we proved that stochastic matrix factorization (SMF) and latent Dirichlet allocation (LDA) have the same real log canonical threshold (RLCT) and that it is smaller than those of regular statistical models. 
By the result, two important facts were derived.  

First, Bayesian generalization errors of SMF and LDA are smaller than
those of regular statistical models, hence LDA and SMF attain the higher generalization performance if Bayesian inference is employed. 

Second, it is well known that RLCTs of learning machines are necessary
for the calculation of sBIC \cite{Drton} and analysis of the replica Monte Carlo method \cite{Nagata2008asymptotic}. Based on the analytic equivalence, LDA and SMF have the same sBIC and the same replica exchange probability. 

In practical applications, we cannot analytically calculate the posterior distribution. 
Our future work will involve numerical experiments and verifying the behavior of our result when the sample size is finite.

\appendix
\section{Proof of Main Theorem}
\label{pf-main}
In this section, we prove Main Theorem using above lemmas.
\begin{proof}(Main Theorem)
Summarizing the terms, we have
\begin{align}\label{Main-1}
\Phi(A,B)&= \| AB-A_0B_0 \|^2  \nonumber \\
&= \sum_{j=1}^N \sum_{i=1}^{M-1} \left\{ \!\sum_{k=1}^{H_0-1} (\!a_{ik}b_{kj} {-}a^0_{ik}b^0_{kj})+a_{iH_0}b_{H_0j}-a^0_{iH_0}b^0_{H_0j}
+\sum_{k=H_0+1}^{H-1}a_{ik}b_{kj} + a_{iH} b_{Hj} \!\right\}^2 \nonumber \\
&\quad + \sum_{j=1}^N \left\{ \!\sum_{k=1}^{H_0-1} (\!a_{Mk}b_{kj} {-}a^0_{Mk}b^0_{kj}){+}a_{MH_0}b_{H_0j}{-}a^0_{MH_0}b^0_{H_0j}
{+}\sum_{k=H_0+1}^{H-1}a_{Mk}b_{kj} {+} a_{MH} b_{Hj}\! \right\}^2\!\!.
\end{align}
Put 
\begin{gather*}
K_{ij}:=\sum_{k=1}^{H_0-1} (a_{ik}b_{kj} -a^0_{ik}b^0_{kj})+a_{iH_0}b_{H_0j}-a^0_{iH_0}b^0_{H_0j}
+\sum_{k=H_0+1}^{H-1}a_{ik}b_{kj} + a_{iH} b_{Hj} , \\
L_j:= \sum_{k=1}^{H_0-1} (a_{Mk}b_{kj} -a^0_{Mk}b^0_{kj})+a_{MH_0}b_{H_0j}-a^0_{MH_0}b^0_{H_0j}
+\sum_{k=H_0+1}^{H-1}a_{Mk}b_{kj} + a_{MH} b_{Hj},
\end{gather*}
then we get
$$\|AB-A_0 B_0\|^2 = \sum_{j=1}^N \sum_{i=1}^{M-1}K_{ij}^2 + \sum_{j=1}^N L_j^2.$$
Using $a_{Mk}=1-\sum_{i=1}^{M-1}a_{ik}$, $b_{Hj}=1-\sum_{k=1}^{H-1} b_{kj}$, $a^0_{Mk}=1-\sum_{i=1}^{M-1}a^0_{ik}$, and $b^0_{H_0j}=1-\sum_{k=1}^{H_0-1} b^0_{kj}$, we have
\begin{gather*}
\sum_{i=1}^{M-1}K_{ij} =\sum_{i=1}^{M-1} \sum_{k=1}^{H-1} (a_{ik}-a_{iH})b_{kj} -\sum_{i=1}^{M-1} \sum_{k=1}^{H_0-1} (a^0_{ik}-a^0_{iH_0})b^0_{kj}+ \sum_{i=1}^{M-1} (a_{iH}-a^0_{iH_0}), \\
L_j=-\sum_{i=1}^{M-1} \sum_{k=1}^{H-1} (a_{ik}-a_{iH})b_{kj} +\sum_{i=1}^{M-1} \sum_{k=1}^{H_0-1}(a^0_{ik}-a^0_{iH_0}) b^0_{kj}
- \sum_{i=1}^{M-1} (a_{iH}-a^0_{iH_0}),
\end{gather*}
thus
$$L_j^2=\left( \sum_{i=1}^{M-1} K_{ij} \right)^2.$$
Therefore 
\begin{align*}
\| AB-A_0B_0\|^2
&= \sum_{j=1}^N \sum_{i=1}^{M-1} K_{ij}^2 +\sum_{j=1}^N L_j^2 \\
&= \sum_{j=1}^N \sum_{i=1}^{M-1} K_{ij}^2 +\sum_{j=1}^N \left( \sum_{i=1}^{M-1} K_{ij} \right)^2. 
\end{align*}
On account of Corollary \ref{idealRLCT2}, we have
$$\| AB-A_0B_0\|^2 \sim \sum_{j=1}^N \sum_{i=1}^{M-1} K_{ij}^2,$$
i.e.
\begin{align*}
&\quad \| AB-A_0B_0\|^2 \\
&\sim \sum_{j=1}^N \sum_{i=1}^{M-1} \left\{ \sum_{k=1}^{H-1} (a_{ik}-a_{iH})b_{kj} - \sum_{k=1}^{H_0-1} (a^0_{ik}-a^0_{iH_0})b^0_{kj}+ (a_{iH}-a^0_{iH_0}) \right\}^2 \\
&= \sum_{j=1}^N \sum_{i=1}^{M-1} \left[ \sum_{k=1}^{H_0-1} \{(a_{ik}-a_{iH})b_{kj}- (a^0_{ik}-a^0_{iH_0})b^0_{kj}\} + \sum_{k=H_0}^{H-1} (a_{ik}-a_{iH})b_{kj}+ (a_{iH}-a^0_{iH_0}) \right]^2.
\end{align*}
$$\mbox{Let} \begin{cases}
a_{ik}=a_{ik}-a_{iH}, & k<H \\
c_i = a_{iH}-a^0_{iH_0}, \\
b_{kj}=b_{kj}
\end{cases}$$
and put $a^0_{ik}=a^0_{ik}-a^0_{iH_0}$. Then we have
\begin{align*}
&\quad \| AB-A_0B_0\|^2 \\
&\sim \sum_{j=1}^N \sum_{i=1}^{M-1} \left[ \sum_{k=1}^{H_0-1} \{(a_{ik}-a_{iH})b_{kj}- (a^0_{ik}-a^0_{iH_0})b^0_{kj}\} + \sum_{k=H_0}^{H-1} (a_{ik}-a_{iH})b_{kj}+ (a_{iH}-a^0_{iH_0}) \right]^2 \\
&= \sum_{j=1}^N \sum_{i=1}^{M-1} \left\{ \sum_{k=1}^{H_0-1} (a_{ik}b_{kj}- a^0_{ik}b^0_{kj}) + \sum_{k=H_0}^{H-1} a_{ik}b_{kj}+ c_i \right\}^2.
\end{align*}
There is a positive constant $C>0$, we have
\begin{align*}
&\quad C \| AB-A_0 B_0 \|^2 \\
&\leqq \sum_{j=1}^N \sum_{i=1}^{M-1}  \left\{ \sum_{k=1}^{H_0-1} (a_{ik}b_{kj}-a^0_{ik}b^0_{kj}) + c_i \right\}^2 
+ \sum_{j=1}^N \sum_{i=1}^{M-1}  \left( \sum_{k=H_0}^{H-1} a_{ik}b_{kj} \right)^2.
\end{align*}
Put 
\begin{gather*}
K_1=\sum_{j=1}^N \sum_{i=1}^{M-1}  \left[ \sum_{k=1}^{H_0-1} (a_{ik}b_{kj}- a^0_{ik}b^0_{kj}) + c_i\right]^2, \\
K_2= \sum_{j=1}^N \sum_{i=1}^{M-1}  \left\{ \sum_{k=H_0}^{H-1} a_{ik}b_{kj} \right\}^2.
\end{gather*}
Let $\bar{\lambda}_1$ be the RLCT of $K_1$ ,$\bar{\lambda}_2$ be the RLCT of $K_2$, and $\lambda$ be the RLCT of $\|AB-A_0B_0\|^2$. The following inequality holds since an RLCT is order isomorphic and $K_1$ and $K_2$ are independent:
$$\lambda \leqq \bar{\lambda}_1 + \bar{\lambda}_2.$$
According to Lemma \ref{lemH} in the case of $H \leftarrow H_0$,
$$\bar{\lambda}_1 \leqq \frac{M-1}{2} + (H_0-1) \frac{M+N-3}{2}.$$
In contrast, there exists a positive constant $D>0$, we have
\begin{align*}
K_2 &= \sum_{j=1}^N \sum_{i=1}^{M-1}  \left( \sum_{k=H_0}^{H-1} a_{ik}b_{kj} \right)^2 \\
&\leqq D\sum_{j=1}^N \sum_{i=1}^{M-1}  \sum_{k=H_0}^{H-1} a_{ik}^2 b_{kj}^2 \\
&\sim \sum_{j=1}^N \sum_{i=1}^{M-1}  \sum_{k=H_0}^{H-1}  a_{ik}^2 b_{kj}^2 \\
&= \sum_{k=H_0}^{H-1} \sum_{j=1}^N \sum_{i=1}^{M-1} a_{ik}^2 b_{kj}^2 \\
&= \sum_{k=H_0}^{H-1} \left( \sum_{j=1}^N b_{kj}^2 \right) \left( \sum_{i=1}^{M-1} a_{ik}^2 \right).
\end{align*}
The RLCT of the last term becomes a sum of each ones about $k$. Considering blowing-ups of variables $\{ a_{ik} \}$ and $\{ b_{kj} \}$ for each $k$, we obtain
$$\bar{\lambda}_2 \leqq \frac{(H-H_0)\min\{ M-1,N \}}{2}.$$

Using the above inequalities about the RLCTs, we have
\begin{align*}
\lambda &\leqq \bar{\lambda}_1 + \bar{\lambda}_2 \\
&\leqq \frac{M-1}{2} + (H_0-1) \frac{M+N-3}{2} + \frac{(H-H_0)\min\{ M-1,N \}}{2}.
\end{align*}
$$\therefore \quad \lambda \leqq \frac{1}{2}\left[
M-1+(H_0-1)(M+N-3)+(H-H_0)\min\{M-1,N\}
\right].$$
\end{proof}

\section{Proof of Lemmas}
\label{pf-lem}
In this section, we prove the four lemmas introduced in Section 3: Lemma \ref{lemH1}, \ref{lemH21}, \ref{lemH22}, and \ref{lemH}.

First,  Lemma \ref{lemH1} is proved.
\begin{proof}(Lemma \ref{lemH1})
We set $A=(a_{i})_{i=1}^{M}$,$B^T=(1)_{j=1}^{N}$,$A_0=(a^0_{i})^M$,$B_0^T=(1)_{j=1}^N$, then
\begin{align*}
\lVert AB-A_0B_0 \rVert^2 &=\sum_{i=1}^M N(a_i-a^0_i)^2 \\
&= \sum_{i=1}^{M-1} N(a_i-a^0_i)^2 + N \left( 1-\sum_{i=1}^{M-1} a_i -1 + \sum_{i=1}^{M-1} a^0_i \right)^2 \\
&= \sum_{i=1}^{M-1} N(a_i-a^0_i)^2 + N \left\{ \sum_{i=1}^{M-1} (a_i - a^0_i) \right\}^2.
\end{align*}
Using Corollary \ref{idealRLCT2}, $\sum_{i=1}^{M-1} (a_i - a^0_i) \in \langle a_1-a^0_1, \ldots, a_{M-1}-a^0_{M-1} \rangle$ causes that
$$\lVert AB-A_0B_0 \rVert^2 \sim \sum_{i=1}^{M-1} N(a_i-a^0_i)^2.$$
As an RLCT is not changed by any constant factor, all we have to do is calculating an RLCT of 
$$\sum_{i=1}^{M-1} (a_i-a^0_i)^2$$
and this has no singularity. Thus, the RLCT equals to a half of the parameter dimension:
$$\lambda = \frac{M-1}{2}.$$
\end{proof}

Second, Lemma \ref{lemH21} is derived.
\begin{proof}(Lemma \ref{lemH21})
We set $A_0=(a^0_{i})^M$,$B_0^T=(1)_{j=1}^N$.
\begin{align*}
&\quad AB-A_0B_0 \\
&=\left( \begin{matrix}
a_{11} & a_{12} \\
\vdots & \vdots \\
a_{(M-1)1} & a_{(M-1)2} \\
a_{M1} & a_{M2}
\end{matrix} \right)
\left( \begin{matrix}
b_1 & \ldots & b_N \\
1-b_1 & \ldots & 1-b_N
\end{matrix} \right)
-\left( \begin{matrix}
a^0_1 \\
\vdots \\
a^0_{M-1} \\
a^0_{M} 
\end{matrix} \right)
\left( \begin{matrix}
1 & \ldots & 1
\end{matrix} \right) \\
&=\left( \begin{matrix}
(a_{11}-a_{12})b_j +a_{12}-a^0_1 \\
\vdots \\
(a_{(M-1)1}-a_{(M-1)2})b_j+a_{(M-1)2}-a^0_{M-1} \\
(a_{M1}-a_{M2})b_j+a_{M2}-a^0_{M} \\
\end{matrix} \right)_{j=1}^N \\
&=\left( \begin{matrix}
(a_{11}-a_{12})b_j+a_{12}-a^0_1 \\
\vdots \\
(a_{(M-1)1}-a_{(M-1)2})b_j+a_{(M-1)2}-a^0_{M-1} \\
-\sum_{i=1}^{M-1}(a_{i1}-a_{i2})b_j-\sum_{i=1}^{M-1}(a_{i2}-a^0_{i}) \\
\end{matrix} \right)_{j=1}^N .
\end{align*}
Thus,
\begin{align*}
\lVert AB-A_0B_0 \rVert^2 &= \sum_{j=1}^N \left( \sum_{i=1}^{M-1} \{(a_{i1}-a_{i2})b_j +a_{i2}-a^0_i\}^2 \right. \\
&\quad \left. + \left[ \sum_{i=1}^{M-1} \{(a_{i1}-a_{i2})b_j +a_{i2}-a^0_i\} \right]^2 \right).
\end{align*}
Put $I=\langle \{(a_{i1}-a_{i2})b_j +a_{i2}-a^0_i\}_{i=1}^{M-1} \rangle$. Because of Corollary \ref{idealRLCT} and 
$$ \sum_{i=1}^{M-1} \{(a_{i1}-a_{i2})b_j +a_{i2}-a^0_i\} \in I,$$
we get 
$$\lVert AB-A_0B_0 \rVert^2 \sim \sum_{j=1}^N  \sum_{i=1}^{M-1} \{(a_{i1}-a_{i2})b_j +a_{i2}-a^0_i\}^2.$$
$$\mbox{Let} \begin{cases}
a_{i}=a_{i1}-a_{i2}, \\
a_{i2} = a_{i2}, \\
b_j=b_j
\end{cases}.$$
Then we get
$$\sum_{j=1}^N \sum_{i=1}^{M-1} \{(a_{i1}-a_{i2})b_j +a_{i2}-a^0_i\}^2=\sum_{j=1}^N \sum_{i=1}^{M-1} \{a_{i}b_j +a_{i2}-a^0_i\}^2.$$
Moreover, 
$$\mbox{Let} \begin{cases}
a_{i}=a_{i}, \\
b_j=b_j ,\\
c_{i} = a_{i2}-a^0_i
\end{cases}$$
and
$$\sum_{j=1}^N \sum_{i=1}^{M-1} \{a_{i}b_j +a_{i2}-a^0_i\}^2=\sum_{j=1}^N \sum_{i=1}^{M-1} \{a_{i}b_j +c_{i}\}^2$$
holds.
$$\mbox{Let} \begin{cases}
a_{i}=a_{i}, \\
b_j=b_j ,\\
x_{i} = a_{i}b_1+c_i, \\
\end{cases}.$$
If $j>1$, then we have $a_i b_j +c_i = x_i -a_i b_1 +a_i b_j$ and obtain
$$\sum_{j=1}^N \sum_{i=1}^{M-1} \{a_{i}b_j +c_{i}\}^2=\sum_{i=1}^{M-1} \left[ x_i^2 + \sum_{j=2}^{N} \{x_i-(a_i b_1-a_i b_j)\}^2 \right].$$
Consider the following generated ideal:
$$J:=\left\langle (x_i)_{i=1}^{M-1}, (a_i b_1 -a_i b_j)_{(i,j)=(1,2)}^{(M-1,N)} \right\rangle.$$
We expand the square terms
$$\{x_i-(a_i b_1-a_i b_j)\}^2=x_i^2 + (a_i b_1-a_i b_j)^2 -2 x_i (a_i b_1 -a_i b_j)$$
and $x_i (a_i b_1 - a_i b_j) \in J$ holds. Hence, owing to Corollary \ref{idealRLCT2}, we have
\begin{align*}
\lVert AB-A_0 B_0 \rVert^2 &\sim \sum_{i=1}^{M-1} \left[ x_i^2 + \sum_{j=2}^{N} \{x_i-(a_i b_1-a_i b_j)\}^2 \right] \\
&\sim \sum_{i=1}^{M-1} \left[ x_i^2 + \sum_{j=2}^{N} (a_i b_1-a_i b_j)^2 \right] \\
&= \sum_{i=1}^{M-1} \left[ x_i^2 + \sum_{j=2}^{N} a_i^2(b_j-b_1)^2 \right].
\end{align*}
$$\mbox{Let} \begin{cases}
a_{i}=a_i, \\
b_1=b_1, \\
b_j=b_j-b_1, & (j>1) \\
x_i =x_i
\end{cases},$$
then we have
\begin{align*}
\lVert AB-A_0 B_0 \rVert^2 &\sim \sum_{i=1}^{M-1} \left\{ x_i^2 + \sum_{j=2}^{N} a_i^2(b_j-b_1)^2 \right\} \\
&=\sum_{i=1}^{M-1} \left( x_i^2 + \sum_{j=2}^{N} a_i^2b_j^2 \right) \\
&=\sum_{i=1}^{M-1} x_i^2 +\sum_{i=1}^{M-1} \sum_{j=2}^N a_i^2 b_j^2 \\
&=\sum_{i=1}^{M-1} x_i^2 +\left( \sum_{i=1}^{M-1}a_i^2 \right) \left( \sum_{j=2}^N b_j^2 \right).
\end{align*}
Since $a_i$,$b_j$,$x_i$ are independent variables for each, we consider blowing-ups of them and get
\begin{align*}
\lambda &= \frac{M-1}{2} + \min \left\{ \frac{M-1}{2}, \frac{N-1}{2} \right\} 
= \min \left\{ M-1,\frac{M+N-2}{2} \right\}.
\end{align*}
Therefore,
$$\lambda =\begin{cases}
M-1 & (M \geqq N) \\
\frac{M+N-2}{2} & (M<N)
\end{cases}.$$
\end{proof}

Third, we prove Lemma \ref{lemH22}.

\begin{proof}(Lemma \ref{lemH22})

\begin{align*}
&\quad AB-A_0B_0 \\
&=\left( \begin{matrix}
a_{11} & a_{12} \\
\vdots & \vdots \\
a_{(M-1)1} & a_{(M-1)2} \\
a_{M1} & a_{M2}
\end{matrix} \right)\!\!\!
\left( \begin{matrix}
b_1 & \!\ldots\! & b_N \\
1{-}b_1 & \!\ldots\! & 1{-}b_N
\end{matrix} \right)
-\left( \begin{matrix}
a^0_{11} & a^0_{12} \\
\vdots & \vdots \\
a^0_{(M-1)1} & a^0_{(M-1)2} \\
a^0_{M1} & a^0_{M2}
\end{matrix} \right)\!\!\!
\left( \begin{matrix}
b^0_1 & \ldots & b^0_N \\
1{-}b^0_1 & \ldots & 1{-}b^0_N
\end{matrix} \right) \\
&=\left( \begin{matrix}
(a_{11}-a_{12})b_j-(a^0_{11}-a^0_{12})b^0_j+a_{12}-a^0_1 \\
\vdots \\
(a_{(M-1)1}-a_{(M-1)2})b_j -(a^0_{(M-1)1}-a^0_{(M-1)2})b^0_j+a_{(M-1)2}-a^0_{M-1} \\
(a_{M1}-a_{M2})b_j-(a^0_{M1}-a^0_{M2})b^0_j+a_{M2}-a^0_{M} \\
\end{matrix} \right)_{j=1}^N\\
&=\left( \begin{matrix}
(a_{11}-a_{12})b_j-(a^0_{11}-a^0_{12})b^0_j+a_{12}-a^0_1 \\
\vdots \\
(a_{(M-1)1}-a_{(M-1)2})b_j -(a^0_{(M-1)1}-a^0_{(M-1)2})b^0_j+a_{(M-1)2}-a^0_{M-1} \\
-\sum_{i=1}^{M-1}\{(a_{i1}-a_{i2})b_1-(a^0_{i1}-a^0_{i2})b^0_j\}-\sum_{i=1}^{M-1}(a_{i2}-a^0_{i}) \\
\end{matrix} \right)_{j=1}^N.
\end{align*}
Then we have
\begin{align*}
\lVert AB-A_0B_0 \rVert^2 &= \sum_{j=1}^N \left( \sum_{i=1}^{M-1} \{(a_{i1}-a_{i2})b_j -(a^0_{i1}-a^0_{i2})b^0_j+a_{i2}-a^0_i\}^2 \right. \\
&\quad \left. + \left[ \sum_{i=1}^{M-1} \{(a_{i1}-a_{i2})b_j-(a^0_{i1}-a^0_{i2})b^0_j +a_{i2}-a^0_i\} \right]^2 \right).
\end{align*}

Put $I=\langle \{(a_{i1}-a_{i2})b_j -(a^0_{i1}-a^0_{i2})b^0_j+a_{i2}-a^0_i\}_{i=1}^{M-1} \rangle$. 
Because of Corollary \ref{idealRLCT} and 
$$ \sum_{i=1}^{M-1} \{(a_{i1}-a_{i2})b_j -(a^0_{i1}-a^0_{i2})b^0_j +a_{i2}-a^0_i\} \in I,$$
we get 
$$\lVert AB-A_0B_0 \rVert^2 \sim \sum_{j=1}^N \sum_{i=1}^{M-1} \{(a_{i1}-a_{i2})b_j -(a^0_{i1}-a^0_{i2})b^0_j+a_{i2}-a^0_i\}^2.$$
$$\mbox{Let} \begin{cases}
a_{i}=a_{i1}-a_{i2}, \\
a_{i2} = a_{i2}, \\
b_j=b_j
\end{cases}$$
and put $a^0_{i}=a^0_{i1}-a^0_{i2}$.
Then we get
$$\sum_{j=1}^N \sum_{i=1}^{M-1} \{(a_{i1}-a_{i2})b_j -(a^0_{i1}-a^0_{i2})b^0_j+a_{i2}-a^0_i\}^2 =\sum_{j=1}^N \sum_{i=1}^{M-1} \{a_{i}b_j -a^0_{i}b^0_j+a_{i2}-a^0_i\}^2.$$
Moreover, 
$$\mbox{Let} \begin{cases}
a_{i}=a_{i}, \\
b_j=b_j ,\\
c_{i} = a_{i2}-a^0_i
\end{cases}$$
and
$$\sum_{j=1}^N  \sum_{i=1}^{M-1} \{a_{i}b_j-a^0_{i}b^0_j +a_{i2}-a^0_i\}^2 =\sum_{j=1}^N \sum_{i=1}^{M-1} \{a_{i}b_j -a^0_{i}b^0_j+c_{i}\}^2 $$
holds.
$$\mbox{Let} \begin{cases}
a_{i}=a_{i}, \\
b_j=b_j ,\\
x_{i} = a_{i}b_1-a^0_{i}b^0_1+c_i, \\
\end{cases}.$$
If $j>1$, then we have $a_i b_j -a^0_{i}b^0_j+c_i = x_i -a_i b_1 +a^0_i b^0_1+a_i b_j-a^0_{i}b^0_j$ and obtain
$$\sum_{j=1}^N \sum_{i=1}^{M-1} \{a_{i}b_j -a^0_i b^0_j+c_{i}\}^2 =\sum_{i=1}^{M-1}  x_i^2 + \sum_{j=2}^{N} \sum_{i=1}^{M-1}\{x_i-(a_i b_1 -a^0_i b^0_1-a_i b_j +a^0_i b^0_j)\}^2.$$
Consider the following generated ideal:
$$J:=\left\langle (x_i)_{i=1}^{M-1}, (a_i b_1 -a^0_i b^0_1-a_i b_j +a^0_i b^0_j)_{(i,j)=(1,2)}^{(M-1,N)} \right\rangle.$$
We expand the square terms
$$\{x_i-(a_i b_1 -a^0_i b^0_1-a_i b_j +a^0_i b^0_j)\}^2=x_i^2 + (a_i b_1 -a^0_i b^0_1-a_i b_j +a^0_i b^0_j)^2 -2 x_i (a_i b_1 -a^0_i b^0_1-a_i b_j +a^0_i b^0_j)$$
and $x_i (a_i b_1 -a^0_i b^0_1-a_i b_j +a^0_i b^0_j) \in J$ holds. Hence, owing to Corollary \ref{idealRLCT2}, we have
\begin{align*}
\lVert AB-A_0 B_0 \rVert^2
&\sim \sum_{i=1}^{M-1} x_i^2 + \sum_{j=2}^{N} \sum_{i=1}^{M-1}\{x_i-(a_i b_1 -a^0_i b^0_1-a_i b_j +a^0_i b^0_j)\}^2 \\
&\sim \sum_{i=1}^{M-1} \left\{ x_i^2 + \sum_{j=2}^{N} (a_i b_1 -a^0_i b^0_1-a_i b_j +a^0_i b^0_j)^2 \right\} \\
&= \sum_{i=1}^{M-1} \left[ x_i^2 + \sum_{j=2}^{N} \{a_i(b_j-b_1)-a^0_i(b^0_j-b^0_1)\}^2 \right].
\end{align*}
$$\mbox{Let} \begin{cases}
a_{i}=a_i, \\
b_1=b_1, \\
b_j=b_j-b_1, & (j>1) \\
x_i =x_i
\end{cases}$$
and put $b^0_j=b^0_j-b^0_1$, then we have
\begin{align*}
\lVert AB-A_0 B_0 \rVert^2 &\sim \sum_{i=1}^{M-1} \left[ x_i^2 + \sum_{j=2}^{N} \{a_i(b_j-b_1)-a^0_i(b^0_j-b^0_1)\}^2 \right] \\
&=\sum_{i=1}^{M-1} \left\{ x_i^2 + \sum_{j=2}^{N} (a_i b_j - a^0_i b^0_j)^2 \right\} \\
&=\sum_{i=1}^{M-1} x_i^2 +\sum_{i=1}^{M-1} \sum_{j=2}^N  (a_i b_j - a^0_i b^0_j)^2
\end{align*}

Let $f_{ij}$ be  $a_i b_j - a^0_i b^0_j$. If $\lVert AB-A_0B_0 \rVert^2 =0$, $f_{ij}=0$. Hence, $a_i \ne 0$ and $b_j \ne 0$. Owing to Proposition \ref{idealNMF}
$$\sum_{i=1}^{M-1} \sum_{j=2}^N f_{ij}^2 \sim \sum_{i=2}^{M-1} f_{i1}^2 + \sum_{j=3}^N f_{1j}^2 +f_{12}^2,$$
we have
\begin{align}
\label{2ndlemH22}
\lVert AB-A_0 B_0 \rVert^2
&\sim \sum_{i=1}^{M-1} x_i^2 +\sum_{i=1}^{M-1} \sum_{j=2}^N  (a_i b_j - a^0_i b^0_j)^2 \nonumber \\
&= \sum_{i=1}^{M-1} x_i^2 +\sum_{i=1}^{M-1} \sum_{j=2}^N f_{ij}^2 \nonumber \\
&\sim \sum_{i=1}^{M-1} x_i^2 +\left( f_{12}^2+\sum_{i=2}^{M-1} f_{i2}^2 + \sum_{j=3}^N f_{1j}^2 \right).
\end{align}
Thus, all we have to do is calculate an RLCT of the right side. Considering blowing-ups, the RLCT $\lambda_1$ of the first term is equal to $\lambda_1=(M-1)/2$. 
For deriving the RLCT of the second term, we arbitrarily take $i,j(1 \leqq i \leqq M-1, 2 \leqq j \leqq N, i,j \in \mathbb{N})$ and fix them.
$$\mbox{Let} \begin{cases}
a_i = a_i, \\
f_{i2} = a_i b_2 - a^0_i b^0_2, \\
f_{1j} = a_1 b_j -a^0_1 b^0_j, \\
x_i = x_i
\end{cases}$$
and we have that the Jacobi matrix of the above transformation is equal to
\[
  \frac{\partial(a_i ,f_{ij}, x_i)}{\partial(a_i,b_j,x_i)}= \left(
    \begin{array}{ccc}
  \frac{\partial a_i}{\partial a_i} & \frac{\partial f_{ij}}{\partial a_i} & \frac{\partial x_i}{\partial a_i}\\
  \frac{\partial a_i}{\partial b_j} & \frac{\partial f_{ij}}{\partial b_j} & \frac{\partial x_i}{\partial b_j}\\
  \frac{\partial a_i}{\partial x_i} & \frac{\partial f_{ij}}{\partial x_i} & \frac{\partial x_i}{\partial x_i}\\
    \end{array}
  \right)
  = \left(
    \begin{array}{ccc}
  1 & b_j & 0 \\
  0 & a_i & 0 \\
  0 & 0  & 1 \\
    \end{array}
  \right).
\]
Because of
$$\Biggl| \frac{\partial(a_i ,f_{ij}, x_i)}{\partial(a_i,b_j,x_i)} \Biggr|=a_i \ne 0,$$
$g$ is an analytic isomorphism. 
Thus, the RLCT $\lambda_2$ of the second term in eq. (\ref{2ndlemH22}) is equal to
$$\lambda_2 = \frac{M+N-3}{2}. $$
Let $\lambda$ be the RLCT of $\lVert AB-A_0B_0 \rVert^2$. From the above,
$$\lambda=\lambda_1+\lambda_2 = \frac{2M+N-4}{2}.$$ 
\end{proof}

Lastly, we derive the inequality of Lemma \ref{lemH}.
\begin{proof}(Lemma \ref{lemH})
We develop the objective function and obtain
\begin{align}
\label{lemHideal1}
&\quad \| AB-A_0B_0 \|^2 \nonumber \\ 
&= \sum_{i=1}^M \sum_{j=1}^N (a_{i1}b_{1j} + ...+ a_{iH}b_{Hj} - a^0_{i1}b^0_{1j} - a^0_{iH}b^0_{Hj})^2 \nonumber \\
&= \sum_{j=1}^N \sum_{i=1}^M \Biggl\{ \sum_{k=1}^{H} (a_{ik}b_{kj} - a^0_{ik}b^0_{kj}) \Biggr\}^2 \nonumber \\
&= \sum_{j=1}^N \sum_{i=1}^{M-1} \Biggl\{ \sum_{k=1}^{H} (a_{ik}b_{kj} - a^0_{ik}b^0_{kj}) \Biggr\}^2 + \sum_{j=1}^N \Biggl\{ \sum_{k=1}^{H} (a_{Mk}b_{kj} - a^0_{Mk}b^0_{kj}) \Biggr\}^2.
\end{align}
Expand the second term in Eq. (\ref{lemHideal1}) by using $a_{Mk}=1-\sum_{i=1}^{M-1}a_{ik}$, $b_{Hj}=1-\sum_{k=1}^{H-1} b_{kj}$, $a^0_{Mk}=1-\sum_{i=1}^{M-1}a^0_{ik}$, and $b^0_{Hj}=1-\sum_{k=1}^{H-1} b^0_{kj}$, then we have
\begin{align*}
&\quad \sum_{j=1}^N \Biggl\{ \sum_{k=1}^{H} (a_{Mk}b_{kj} - a^0_{Mk}b^0_{kj}) \Biggr\}^2 \\
&= \sum_{j=1}^N \Biggl\{ \sum_{k=1}^{H-1} (a_{Mk}b_{kj} - a^0_{Mk}b^0_{kj}) + (a_{MH}b_{Hj} - a^0_{MH}b^0_{Hj}) \Biggr\}^2 \\
&=\sum_{j=1}^N \Biggl(- \sum_{i=1}^{M-1} \sum_{k=1}^{H-1} a_{ik}b_{kj} + \sum_{i=1}^{M-1} \sum_{k=1}^{H-1}  a^0_{ik}b^0_{kj} \\
&\quad -\sum_{i=1}^{M-1}a_{iH} +\sum_{i=1}^{M-1}\sum_{k=1}^{H-1} a_{iH}b_{kj}+ \sum_{i=1}^{M-1}a^0_{iH} -\sum_{i=1}^{M-1} \sum_{k=1}^{H-1} a^0_{iH}b^0_{kj} \Biggr)^2 {=}\!:\! \Phi_2.
\end{align*}
Developing the equation, we have
\begin{align*}
\Phi_2&=\sum_{j=1}^N \Biggl\{- \sum_{i=1}^{M-1} \sum_{k=1}^{H-1} (a_{ik}-a_{iH})b_{kj} + \sum_{i=1}^{M-1} \sum_{k=1}^{H-1}  (a^0_{ik}-a^0_{iH})b^0_{kj} -\sum_{i=1}^{M-1}(a_{iH}-a^0_{iH})  \Biggr\}^2 \\
&=\sum_{j=1}^N \Biggl[\sum_{i=1}^{M-1} \sum_{k=1}^{H-1} \{(a_{ik}{-}a_{iH})b_{kj} {-} (a^0_{ik}{-}a^0_{iH})b^0_{kj}\}{+}\sum_{i=1}^{M-1}(a_{iH}{-}a^0_{iH})  \Biggr]^2.
\end{align*}

On the other hand, the first term of equation (\ref{lemHideal1}) is equal to
\begin{align*}
&\quad \sum_{j=1}^N \sum_{i=1}^{M-1} \Biggl\{ \sum_{k=1}^{H} (a_{ik}b_{kj} - a^0_{ik}b^0_{kj}) \Biggr\}^2 \\
&= \sum_{j=1}^N \sum_{i=1}^{M-1} \Biggl\{ \sum_{k=1}^{H-1} (a_{ik}b_{kj} - a^0_{ik}b^0_{kj}) + (a_{iH}b_{Hj} - a^0_{iH}b^0_{Hj}) \Biggr\}^2 \\
&= \sum_{j=1}^N \sum_{i=1}^{M-1} \Biggl\{ \sum_{k=1}^{H-1} (a_{ik}b_{kj} - a^0_{ik}b^0_{kj}) 
+ a_{iH}-\sum_{k=1}^{H-1} a_{iH}b_{kj} - a^0_{iH} + \sum_{k=1}^{H-1} a^0_{iH} b^0_{kj} \Biggr\}^2 \\
&= \sum_{j=1}^N \sum_{i=1}^{M-1} \Biggl\{ \sum_{k=1}^{H-1} (a_{ik}b_{kj} - a^0_{ik}b^0_{kj}) 
+ (a_{iH} - a^0_{iH})-\sum_{k=1}^{H-1} (a_{iH}b_{kj} - a^0_{iH} b^0_{kj}) \Biggr\}^2 \\
&= \sum_{j=1}^N \sum_{i=1}^{M-1} \Biggl[(a_{iH} - a^0_{iH})+\sum_{k=1}^{H-1} \{(a_{ik}-a_{iH})b_{kj} - (a^0_{ik}-a^0_{iH}) b^0_{kj}\} \Biggr]^2 .
\end{align*}
Consider the following ideal:
$$I=\left\langle (a_{iH}-a^0_{iH})_{i=1}^{M-1} , \{(a_{ik}-a_{iH})b_{kj} - (a^0_{ik}-a^0_{iH}) b^0_{kj}\}_{(i,j,k)=(1,1,1)}^{(M-1,N,H-1)} \right\rangle.$$
Since we have
\begin{align*}
\lVert AB-A_0B_0 \rVert^2 &= \sum_{j=1}^N \sum_{i=1}^{M-1} \Biggl[(a_{iH} - a^0_{iH})+\sum_{k=1}^{H-1} \{(a_{ik}-a_{iH})b_{kj} - (a^0_{ik}-a^0_{iH}) b^0_{kj}\} \Biggr]^2 \\
&\quad + \sum_{j=1}^N \Biggl[\sum_{i=1}^{M-1} \sum_{k=1}^{H-1} \{(a_{ik}{-}a_{iH})b_{kj} {-} (a^0_{ik}{-}a^0_{iH})b^0_{kj}\}{+}\sum_{i=1}^{M-1}(a_{iH}{-}a^0_{iH})  \Biggr]^2
\end{align*}
and
$$\forall j, \sum_{i=1}^{M-1} \sum_{k=1}^{H-1} \{(a_{ik}{-}a_{iH})b_{kj} {-} (a^0_{ik}{-}a^0_{iH})b^0_{kj}\}{+}\sum_{i=1}^{M-1}(a_{iH}{-}a^0_{iH}) \in I,$$
thus Corollary \ref{idealRLCT2} causes
$$\lVert AB-A_0B_0 \rVert^2 \sim \sum_{j=1}^N \sum_{i=1}^{M-1} \Biggl[(a_{iH} - a^0_{iH})+\sum_{k=1}^{H-1} \{(a_{ik}-a_{iH})b_{kj} - (a^0_{ik}-a^0_{iH}) b^0_{kj}\} \Biggr]^2.$$
We transform the coordinate like the proof of Lemma \ref{lemH22} for resolution singularity of the above polynomial.
$$\mbox{Let} \begin{cases}
a_{ik}=a_{ik}-a_{iH}, & (k<H) \\
a_{iH}=a_{iH}, \\
b_{kj}=b_{kj}, \\
\end{cases}$$
and put $a^0_{ik}=a^0_{ik}-a^0_{iH}$,
\begin{align*}
&\quad \sum_{j=1}^N \sum_{i=1}^{M-1} \Biggl[(a_{iH} - a^0_{iH})+\sum_{k=1}^{H-1} \{(a_{ik}-a_{iH})b_{kj} - (a^0_{ik}-a^0_{iH}) b^0_{kj}\} \Biggr]^2 \\
&=\sum_{j=1}^N \sum_{i=1}^{M-1} \Biggl[(a_{iH} - a^0_{iH})+\sum_{k=1}^{H-1} (a_{ik}b_{kj} {-} a^0_{ik}b^0_{kj}) \Biggr]^2.
\end{align*}
$$\mbox{Let} \begin{cases}
a_{ik}=a_{ik},\\
b_{kj}=b_{kj}, \\
c_i=a_{iH}-a^0_{iH}
\end{cases}.$$
Then we obtain
\begin{align*}
&\quad \sum_{j=1}^N \sum_{i=1}^{M-1} \Biggl[(a_{iH} - a^0_{iH})+\sum_{k=1}^{H-1} (a_{ik}b_{kj} {-} a^0_{ik}b^0_{kj}) \Biggr]^2 \\
&=\sum_{j=1}^N \sum_{i=1}^{M-1} \Biggl[c_i+\sum_{k=1}^{H-1} (a_{ik}b_{kj} {-} a^0_{ik}b^0_{kj}) \Biggr]^2 \\
&=\sum_{i=1}^{M-1} \sum_{j=1}^{N} \Biggl[c_i+\sum_{k=1}^{H-1} (a_{ik}b_{kj} {-} a^0_{ik}b^0_{kj}) \Biggr]^2 \\
&=\sum_{i=1}^{M-1} \left\{ \sum_{k=1}^{H-1} (a_{ik}b_{k1} {-} a^0_{ik}b^0_{k1}){+}c_i \right\}^2 + \sum_{j=2}^N \sum_{i=1}^{M-1} \left\{ \sum_{k=1}^{H-1} (a_{ik}b_{kj} {-} a^0_{ik}b^0_{kj}){+}c_i \right\}^2
\end{align*}
In addition, 
$$\mbox{Let} \begin{cases}
a_{ik}=a_{ik},\\
b_{kj}=b_{kj}, \\
x_i=\sum_{k=1}^{H-1} (a_{ik}b_{k1} - a^0_{ik}b^0_{k1}) + c_i
\end{cases}.$$
If $j>1$, then we have 
\begin{align*}
\sum_{k=1}^{H-1} (a_{ik}b_{kj} {-} a^0_{ik}b^0_{kj}){+}c_i &= x_i -\sum_{k=1}^{H-1} (a_{ik}b_{k1} - a^0_{ik}b^0_{k1}) +\sum_{k=1}^{H-1} (a_{ik}b_{kj} {-} a^0_{ik}b^0_{kj})\\
&=x_i +\sum_{k=1}^{H-1}\{ (a_{ik}b_{kj} {-} a^0_{ik}b^0_{kj})- (a_{ik}b_{k1} - a^0_{ik}b^0_{k1})\}
\end{align*}
and obtain
\begin{align*}
&\quad \sum_{j=1}^N \sum_{i=1}^{M-1} \left\{\sum_{k=1}^{H-1} (a_{ik}b_{kj} {-} a^0_{ik}b^0_{kj}){+}c_i \right\}^2 \\
&=\sum_{i=1}^{M-1}  x_i^2 + \sum_{j=2}^{N}\sum_{i=1}^{M-1} \left[ x_i +\sum_{k=1}^{H-1}\{ (a_{ik}b_{kj} {-} a^0_{ik}b^0_{kj})- (a_{ik}b_{k1} - a^0_{ik}b^0_{k1})\} \right]^2.
\end{align*}
Put $$g_{ij}:=\sum_{k=1}^{H-1}\{ (a_{ik}b_{kj} {-} a^0_{ik}b^0_{kj})- (a_{ik}b_{k1} - a^0_{ik}b^0_{k1})\}.$$
Consider the following ideal:
$$J:=\left\langle (x_i)_{i=1}^{M-1}, (g_{ij})_{(i,j)=(1,2)}^{(M-1,N)} \right \rangle.$$
We expand the square terms
$$(x_i+g_{ij})^2=x_i^2 + (g_{ij})^2 +2x_i g_{ij}$$
and $x_i g_{ij} \in J$. Hence, owing to Corollary \ref{idealRLCT2}, we get
\begin{align*}
\lVert AB-A_0 B_0 \rVert^2 &\sim \sum_{i=1}^{M-1} x_i^2 + \sum_{j=2}^{N} \sum_{i=1}^{M-1} \{x_i+g_{ij}\}^2 \\
&\sim \sum_{i=1}^{M-1} \left\{ x_i^2 + \sum_{j=2}^{N} (g_{ij})^2 \right\} \\
&= \sum_{i=1}^{M-1} \left( x_i^2 + \sum_{j=2}^{N} \left[ \sum_{k=1}^{H-1}\{ (a_{ik}b_{kj} {-} a^0_{ik}b^0_{kj})- (a_{ik}b_{k1} - a^0_{ik}b^0_{k1})\}\right]^2 \right) \\
&= \sum_{i=1}^{M-1} \left( x_i^2 + \sum_{j=2}^{N} \left[ \sum_{k=1}^{H-1}\{ a_{ik}(b_{kj} {-} b_{k1}) -a^0_{ik}(b^0_{kj}- b^0_{k1})\}\right]^2 \right).
\end{align*}
$$\mbox{Let} \begin{cases}
a_{ik}=a_{ik}, \\
b_{k1}=b_{k1}, \\
b_{kj}=b_{kj}-b_{k1}, & (j>1) \\
x_i =x_i
\end{cases}$$
and put $b^0_{kj}=b^0_{kj}- b^0_{k1}$, then we have
\begin{align*}
\lVert AB-A_0 B_0 \rVert^2 &\sim \sum_{i=1}^{M-1} \left( x_i^2 + \sum_{j=2}^{N} \left[ \sum_{k=1}^{H-1}\{ a_{ik}(b_{kj} {-} b_{k1}) -a^0_{ik}(b^0_{kj}- b^0_{k1})\}\right]^2 \right) \\
&=\sum_{i=1}^{M-1} \left[ x_i^2 + \sum_{j=2}^{N} \left\{ \sum_{k=1}^{H-1}(a_{ik}b_{kj} -a^0_{ik}b^0_{kj})\right\}^2 \right] \\
&=\sum_{i=1}^{M-1} x_i^2 +\sum_{i=1}^{M-1} \sum_{j=2}^{N} \left\{ \sum_{k=1}^{H-1}(a_{ik}b_{kj} -a^0_{ik}b^0_{kj})\right\}^2.
\end{align*}
There exists a positive constant $C>0$, we have
\begin{align*}
\lVert AB-A_0B_0 \rVert^2 
&\sim \sum_{i=1}^{M-1} x_i^2 +\sum_{i=1}^{M-1} \sum_{j=2}^{N} \left\{ \sum_{k=1}^{H-1}(a_{ik}b_{kj} -a^0_{ik}b^0_{kj})\right\}^2 \\
&\leqq \sum_{i=1}^{M-1} x_i^2 +C\sum_{i=1}^{M-1} \sum_{j=2}^{N} \sum_{k=1}^{H-1}(a_{ik}b_{kj} -a^0_{ik}b^0_{kj})^2 \\
&\sim \sum_{i=1}^{M-1} x_i^2 +\sum_{i=1}^{M-1} \sum_{j=2}^{N} \sum_{k=1}^{H-1}(a_{ik}b_{kj} -a^0_{ik}b^0_{kj})^2 \\
&= \sum_{i=1}^{M-1} x_i^2 + \sum_{k=1}^{H-1} \sum_{i=1}^{M-1} \sum_{j=2}^N (a_{ik} b_{kj} -a^0_{ik} b^0_{kj})^2.
\end{align*}
We blow-up the coordinate like the proof of Lemma \ref{lemH22} for resolution singularity in
$$\sum_{i=1}^{M-1} x_i^2 + \sum_{k=1}^{H-1} \sum_{i=1}^{M-1} \sum_{j=2}^N (a_{ik} b_{kj} -a^0_{ik} b^0_{kj})^2.$$
Let $\bar{\lambda}_1$ be the RLCT of the first term and $\bar{\lambda}_2$ be the RLCT of the second term.
It is immediately proved that $\bar{\lambda}_1$ is equal to $(M-1)/2$. 
For deriving the RLCT of the second term $\bar{\lambda}_2$, we use the result of Lemma \ref{lemH22}: the RLCT of
$\sum_{i=1}^{M-1} \sum_{j=2}^N (a_{ik} b_{kj} -a^0_{ik} b^0_{kj})^2$ is equal to $(M+N-3)/2$. Thus we have
$$\bar{\lambda}_2 = (H-1) \frac{M+N-3}{2}.$$
Let $\lambda$ be the RLCT of $\| AB-A_0B_0\|^2$. In general, an RLCT is order isomorphic, therefore we get
\begin{align*}
\lambda 
&\leqq \bar{\lambda_1} + \bar{\lambda}_2 \\
&= \frac{M-1}{2} + (H-1) \frac{M+N-3}{2} \\
&= \frac{M-1+(H-1)(M+N-3)}{2}. 
\end{align*}
\end{proof}


\begin{acknowledgements}
This research was partially supported by 
NTT DATA Mathematical Systems Inc..
The authors would like to thank the editor and the reviewers for comments to improve this paper.
\end{acknowledgements}

\section*{Conflict of Interest}

Conflict of interests in our research are as below: 
\begin{itemize}
\item The first author is a member of NTT DATA Mathematical Systems Inc., Japan.
\item The second author is a member of Tokyo Institute of Technology, Japan.
\item The authors declare that they do not have any other conflicts of interest. 
\end{itemize}

\bibliographystyle{spbasic}      
\bibliography{bibs-full}   

%
%

\end{document}